\documentclass[letterpaper,11pt]{amsart}

\usepackage{graphicx,color,epsfig}
\usepackage{amsfonts,amsmath,amssymb,amsthm}
\usepackage{hyperref}
\usepackage{comment}

\newtheorem{thm}{Theorem}[section]
\newtheorem{prop}[thm]{Proposition}
\newtheorem{lem}[thm]{Lemma}
\newtheorem{cor}[thm]{Corollary}
\newtheorem{con}[thm]{Conjecture}

\newtheorem{asm}{Assumption}

\theoremstyle{remark}
\newtheorem{rem}[thm]{Remark}

\theoremstyle{definition}
\newtheorem{defn}{Definition}

\newcommand{\ra}{\rightarrow}

\newcommand{\N}{\mathbb N}     
\newcommand{\R}{\mathbb R}     
\newcommand{\Z}{\mathbb Z}     

\renewcommand{\a}{\alpha}
\renewcommand{\b}{\beta}

\renewcommand{\d}{\delta}

\newcommand{\e}{\varepsilon}

\newcommand{\s}{\sigma}
\renewcommand{\th}{\theta}

\newcommand{\ind}[1]{ \mathbf{1}_{ \{ #1 \} } } 

\newcommand{\be}{\begin{equation}}
\newcommand{\ee}{\end{equation}}

\newcommand{\w}{\omega}              
\renewcommand{\P}{\mathbb{P}}        
\newcommand{\E}{\mathbb{E}}          

\DeclareMathOperator{\supp}{supp}

\begin{document}

\title{Excited random walks with non-nearest neighbor steps}
\author{Burgess Davis}
\address{Burgess Davis\\ Purdue University \\ 150 N University Street \\ West Lafayette, IN 47907 \\ USA}
\email{bdavis@stat.purdue.edu}
\urladdr{http://www.stat.purdue.edu/~bdavis/}

\author{Jonathon Peterson}
\address{Jonathon Peterson \\  Purdue University \\ Department of Mathematics \\ 150 N University Street \\ West Lafayette, IN 47907 \\ USA}
\email{peterson@math.purdue.edu}
\urladdr{http://www.math.purdue.edu/~peterson}
\thanks{J. Peterson was partially supported by NSA grants H98230-13-1-0266 and H98230-15-1-0049.}


\date{\today}

\subjclass[2000]{Primary 60K35; Secondary 60K37, 60J15}

\begin{abstract}
Let $W$ be an integer valued random variable satisfying $E[W] =: \d \geq 0$ and $P(W<0)>0$, and consider a self-interacting random walk that behaves like a simple symmetric random walk with the exception that on the first visit to any integer $x\in \Z$ the size of the next step is an independent random variable with the same distribution as $W$. We show that this self-interacting random walk is recurrent if $\d\leq 1$ and transient if $\d>1$. This is a special case of our main result which concerns the recurrence and transience of excited random walks (or cookie random walks) with non-nearest neighbor jumps. 
\end{abstract}

\maketitle

\section{Introduction and statement of main results}

This paper concerns the study of one-dimensional excited random walks with possibly non-nearest neighbor jumps.  
Excited random walks (also called cookie random walks) are a model for self-interacting random motion where the transition probabilities are a function of the local time of the walk at the present location. 
While self-interacting random walks are typically very difficult to study, much is known about one-dimensional nearest neighbor excited random walks. 
Under mild assumptions there are explicit criteria for recurrence/transience, ballisticity, and a characterization of the limiting distributions of the excited random walks \cite{zMERW,bsCRWspeed,bsRGCRW,kzPNERW,kmLLCRW}. 
In the current paper, we study a model for excited random walks that allows for jumps that can even be unbounded. 
We will prove a simple criterion for recurrence/transience of the excited random walk that generalizes the known results for nearest neighbor excited random walks. 


There has been some limited study of excited random walks with bounded jumps in dimensions $d\geq 2$. 
The model that we discuss below is inspired in large measure by the generalized excited random walk introduced in \cite{mprvGMDERW}. However, the methods developed in \cite{mprvGMDERW} were limited to $d\geq 2$, and the behavior of nearest neighbor excited random walks is very different in dimensions larger than one. 
In particular, multi-dimensional excited random walks with local drifts contained in a fixed half-plane are known to be transient with a non-zero limiting speed and CLT type limiting distributions \cite{brCLTERW,mprvGMDERW}. In contrast, for one-dimensional excited random walks it is known that the walks can be transient with sublinear speed and with non-Gaussian limiting distributions \cite{bsCRWspeed,bsRGCRW,kmLLCRW}. 

\subsection{Excited random walks with non-nearest neighbor jumps}
We will describe the model of excited random walks with non-nearest neighbor jumps using the terminology of \emph{cookie environments}\footnote{The ``cookie'' terminology will be explained in Remark \ref{rem:Mcookie} below} introduced by Zerner in \cite{zMERW}. 
Let $M_1(\Z)$ denote the space of probability distributions on $\Z$. 
A cookie environment is an element 
\[
 \w = \{ \w_{x,j} \}_{x\in \Z, j\geq 1} \in (M_1(\Z))^{\Z \times \N} =: \Omega. 
\]
Given a cookie environment $\w$, we will construct a random walk $\{X_n\}_{n\geq 0}$ as follows. 
For each $x \in \Z$ and $j\geq 1$, $\w_{x,j} = (\w_{x,j}(z))_{z \in \Z}$ will determine the step distribution for the random walk upon the $j$-th visit to the site $x$. 
More specifically, given any $x\in \Z$ and $\w \in \Omega$, the excited random walk in the cookie environment $\w$ started at $x$ is a stochastic process $\{X_n\}_{n\geq 0}$ with distribution $P_\w^x$ given by $P_\w^x(X_0 = x) = 1$ and 
\begin{equation}\label{ERWdef}
 P_\w^x \left( X_{n+1} = X_n+z \, | \, \s(X_k, \, k\leq n) \right) = \w_{X_n,L_n(X_n)}(z), \quad \text{where } L_n(y) = \sum_{k=0}^n \ind{X_k = y}. 
\end{equation}

In general, we will allow the cookie environments $\w$ to be chosen from a measure $\eta$ on the space of cookie environments $\Omega$. 
That is, using the notation $\w_x = \{\w_{x,j}\}_{j \geq 1}$ for the cookie environment at the site $x\in \Z$, we will assume that $\{\w_x\}_{x\in \Z}$ is a stationary and ergodic sequence of $M_1(\Z)^\N$-valued random variables under the measure $\eta$.  
For a fixed cookie environment, the measure $P_\w^x$ is called the \emph{quenched distribution} of the random walk. 
If we then average the quenched distribution over all environments according to the measure $\eta$ we obtain the \emph{averaged distribution} of the random walk $\P_\eta^x(\cdot) = E_\eta[ P_\w^x(\cdot) ]$, where $E_\eta$ denotes expectation with respect to the measure $\eta$ on cookie environments.  
We note that often we will be interested in the excited random walk starting at $X_0 = 0$ in which case we will use the notation $P_\w$ and $\P_\eta$ in place of $P_\w^0$ and $\P_\eta^0$. Expectations with respect to $P_\w^x$ and $\P_\eta^x$ will be denoted by $E_\w^x$ and $\E_\eta^x$, respectively, where again the superscript $x$ will be omitted when $x=0$. 

For one-dimensional nearest neighbor excited random walks, much of the qualitative behavior of the random walk is determined by a single parameter: the expected total drift of the cookies environment at a site. 
That is, if $d(\mu) = \sum_z z \mu(z)$ denotes the mean of the distribution $\mu \in M_1(\Z)$ (assuming the mean is finite) then the expected total drift of the cookie environment at a site is
\begin{equation}\label{ddef}
 \d = \d(\eta) := E_\eta\left[ \sum_{j\geq 1} d(\w_{0,j}) \right] = E_\eta\left[ \sum_{j\geq 1} \sum_{z \in \Z} z \, \w_{0,j}(z) \right].
\end{equation}
Naturally, for the model as currently stated there is no guarantee that the expectation on the right side of \eqref{ddef} exists. Note that in the case of nearest neighbor excited random walks the formula for $\d$ simplifies to $\d = E_\eta[ \sum_{j\geq 1} (2\w_{0,j}(1) - 1) ] $, and thus the parameter $\d$ is well defined if the measure $\eta$ is restricted to cookie environments with one of the following properties. 
\begin{itemize}
 \item $\w_{x,j}(1) \geq 1/2$ 
for all $x\in \Z$, $j\geq 1$.
 \item There exists an $M<\infty$ such that $\w_{x,j}(1) = 1/2$ 
for all $x\in \Z$, $j>M$. 
\end{itemize}
These properties are commonly referred to as the case of \emph{non-negative cookies} and \emph{finitely many cookies per site}, respectively. 
In either of these cases, it is known that the parameter $\d$ determines the recurrence/transience of the excited random walk \cite{zMERW,kzPNERW}. That is, the excited random walk is recurrent if $\d \in [-1,1]$, transient to the right if $\d>1$ and transient to the left if $\d < -1$. 
Moreover, in the case of finitely many cookies per site, the parameter $\d$ gives a criterion for when the limiting speed of the excited random walk is non-zero \cite{bsCRWspeed,kzPNERW} and also determines the limiting distribution for the excited random walk \cite{bsRGCRW,kmLLCRW,kzPNERW,dkSLRERW}. 

To prepare for our main results, we first need to provide some assumptions on the measure $\eta$ on cookie environments which will ensure the parameter $\d$ in \eqref{ddef} is well defined. 
To this end, let $\mathcal{M}_0 = \{ \mu \in M_1(\Z): \, d(\mu) = 0 \}$ be the subset of probability measures on $\Z$ with zero mean, and let $\mathcal{M}_+ = \{ \mu \in M_1(\Z): \, d(\mu) \in [0,\infty) \}$
be the subset of probability measures on $\Z$ with non-negative finite mean. 
We will assume that our cookie environments $\w$ belong to a fixed subset $\Omega_{M,\mu}^+ \subset \Omega$ defined by 
\[
 \Omega_{M,\mu}^+ = \left( \mathcal{M}_+^M \times \{\mu \}^\N \right)^\Z, \qquad M \geq 1, \, \mu \in \mathcal{M}_0. 
\]
That is, $\w \in \Omega_{M,\mu}^+$ if and only if $d(\w_{x,j}) \geq 0$ for all $x\in\Z$, $j\leq M$ and $\w_{x,j} = \mu$ for all $x\in \Z$ and $j>M$. 
\begin{rem}\label{rem:Mcookie}
The terminology of ``cookie environments'' is helpful in describing the walk in the following way. For any $\w \in \Omega_{M,\mu}^+$ one imagines a stack of $M$ cookies at each site $x$, where the $j$-th cookie in the stack corresponds to the distribution $\w_{x,j} \in M_1(\Z)$. When the random walker reaches a site $x$ for the $j$-th time with $j\leq M$, the walker consumes the $j$-th cookie in the stack which ``excites'' the walker by inducing a jump probability of $\w_{x,j}$ with non-negative drift for the next step. If the walker ever visits a site $x$ more than $M$ times then on every visit after the $M$-th visit there are no cookies remaining at that site to excite the walker and so the next step is chosen according to the fixed distribution $\mu$ with zero mean. 
%
\end{rem}


Throughout the paper we will make the following assumption for the distribution $\eta$ on cookie environments. 

\begin{asm}\label{asmMmu}
 The measure $\eta$ on cookie environments is such that $\{\w_x\}_{x\in\Z}$ is stationary and ergodic and $\eta( \w \in \Omega_{M,\mu}^+) = 1$ for some fixed $M\geq 1$ and $\mu \in \mathcal{M}_0 \backslash \{\d_0\}$. 
\end{asm}

\begin{asm}\label{asmmu}
 The distribution $\mu$ in Assumption \ref{asmMmu} is such that 
\begin{itemize}
 \item $\mu([-B,B]) = 1$ for some $B<\infty$, and
 \item the span of $\supp \mu$ is $1$. That is, the random walk with step distribution $\mu$ is aperiodic. 
\end{itemize}
\end{asm}

\begin{asm}\label{asmic}
 With $M\geq 1$ as in Assumption \ref{asmMmu}, the distribution $\eta$ is such that 
\begin{equation}\label{cookiefm}
 E_\eta\left[ \sum_{j=1}^M \sum_{z \in \Z} |z| \, \w_{0,j}(z) \right] < \infty.
\end{equation}
\end{asm}

\begin{rem}
 Assumptions \ref{asmMmu}--\ref{asmic} are rather mild. The main restriction we are making is that $\eta$ is concentrated on $\Omega_{M,\mu}^+$; that is, $M$ cookies per site and all cookies induce a non-negative drift. 
The requirements that $\mu$ have finite support and the expectation in \eqref{cookiefm} is finite are technical conditions needed for our proofs, though it is not clear if a different proof could remove these restrictions. 
The other conditions in Assumptions \ref{asmMmu}--\ref{asmic} are used to avoid obvious degeneracies. For instance without the aperiodicity condition on $\mu$ in Assumption \ref{asmmu} our characterization of recurrence/transience by $\d$ in Theorem \ref{thm:rt} below is no longer true. Indeed, one can trivially modify a cookie environment satisfying the assumptions in this paper in such a way that the random walk is concentrated on $2\Z$ without changing the recurrence/transience of the walk but so that the value of $\d$ is doubled. 
\end{rem}

\begin{rem}
 Note that \eqref{cookiefm} in Assumption \ref{asmic} implies that the parameter $\d$ as defined in \eqref{ddef} is finite. However, the converse is not true. 
For instance, suppose that $M=1$ and that the distribution $\eta$ on cookie environments is such that the jump distribution of the walk on the first visit to $x\in \Z$ is given by 
\[
 \w_{x,1}(z)
= \begin{cases}
   1/2 & \text{if } |z| = K_x \\
   0 & \text{otherwise},
  \end{cases}
\]
where $\{K_x\}_{x\in\Z}$ is an i.i.d.\ sequence of non-negative integer valued random variables. Then, clearly $d(\w_{x,1}) = 0$ almost surely, but the expectation in \eqref{cookiefm} is equal to $E_\eta[K_0]$ which may be infinite. 
\end{rem}

Let $\mathcal{R}$, $\mathcal{T}_+$ and $\mathcal{T}_-$ be the events defined by 
\[
 \mathcal{R} = \left\{ \sum_{n\geq 0} \ind{X_n = x} = \infty, \, \forall x \in \Z \right\} 
\quad \text{and}\quad
 \mathcal{T}_\pm = \{ \lim_{n\ra\infty} X_n = \pm \infty \}. 
\]
That is, $\mathcal{R}$ is the event that the random walk is recurrent and $\mathcal{T}_+$ (or $\mathcal{T}_-$) is the event that the random walk is transient to the right (or left). 
Since we have assumed that all cookies induce a non-negative drift, it is natural to expect that the walk is either recurrent or transient to the right. Indeed, 
our first main result shows that this is the case and also gives a sufficient condition for the walk to be recurrent. 
\begin{thm}\label{thm:ergrt}
 Let Assumptions \ref{asmMmu}--\ref{asmic} hold. Then, 
\begin{enumerate}
 \item $\P_\eta( \mathcal{R} \cup \mathcal{T}_+ ) = 1$. \label{K01law}
 \item Moreover, if $\d<1$ then $\P_\eta( \mathcal{R} ) = 1$. 
\end{enumerate}
\end{thm}

Theorem \ref{thm:ergrt} gives a sufficient condition for recurrence in terms of the parameter $\d$. 
To improve this to a complete characterization of recurrence/transience in terms of the value of $\d$ we will need a few more assumptions on the distribution $\eta$ on cookie environments. 

\begin{asm}\label{asmiid}
 The sequence $\{\w_x\}_{x \in \Z}$ is i.i.d.\ under the measure $\eta$. 
\end{asm}


\begin{asm}\label{asmell}
$\eta( \w_{0,1}([1,\infty)) > 0 ) = 1$ 
and $E_\eta\left[ \prod_{j=1}^M \sum_{z\leq 0} \w_{0,j}(z) \right] > 0$.
\end{asm}

\begin{rem}
 If Assumption \ref{asmiid} is satisfied we will say that the cookie environment is \emph{spatially i.i.d.}
Assumption \ref{asmell} is a weak ellipticity assumption with the following probabilistic interpretation: 
the random walk always has a positive probability of jumping to the right on the first visit to a site and 
there is a positive probability 
that all $M$ of the cookies at a given site allow for jumps that are non-positive. 
Note that since $\mu$ is a non-degenerate distribution with zero mean, it follows that $\mu([-B,-1])>0$ so that whenever there are no cookies remaining at a site the walk has a positive probability of stepping to the left. 
\end{rem}

Our other main result extends the characterization for recurrence/transience of nearest neighbor excited random walks to non-nearest neighbor excited random walks satisfying the above assumptions. 
\begin{thm}\label{thm:rt}
 Let Assumptions \ref{asmMmu}--\ref{asmell} hold. Then, 
\begin{enumerate}
 \item $\P_\eta(\mathcal{R}) = 1 \iff \d \leq 1$. 
 \item $\P_\eta(\mathcal{T}_+) = 1 \iff \d > 1$. 
\end{enumerate}
\end{thm}

\begin{rem}
 A characterization of recurrence/transience of nearest neighbor excited random walks was proved by Zerner in \cite{zMERW}
 under conditions corresponding only to Assumptions \ref{asmMmu}--\ref{asmic}. 
We note that in the present paper Assumptions \ref{asmiid} and \ref{asmell} are only needed in the proof of Theorem \ref{thm:rt} to prove the 0-1 law for recurrence/transience in Section \ref{sec:01}. 
For nearest neighbor one-dimensional excited random walks this 0-1 law is known to hold for ergodic cookie environments \cite{ABOerg01}, and thus we conjecture that Theorem \ref{thm:rt} can be extended to ergodic cookie environments (possibly under some additional ellipticity assumptions). 
\end{rem}

As noted above, for nearest neighbor excited random walks many other aspects of the random walk are characterized by the parameter $\d$. We conjecture that (under appropriate assumptions) these results can be extended to excited random walks with non-nearest neighbor jumps. In particular, we mention the following conjecture. 
\begin{con}\label{con:LLN}
 Under Assumptions \ref{asmMmu}--\ref{asmell}, there exists a deterministic constant $v_0 \in [0,\infty)$ such that 
\[
 \lim_{n\ra\infty} \frac{X_n}{n} = v_0, \qquad \P_\eta\text{-a.s.}
\]
Moreover, $v_0 > 0$ if and only if $\d>2$. 
\end{con}
\noindent
Some partial progress regarding Conjecture \ref{con:LLN} is already underway, and we hope to address this more fully in a future paper.

The structure of the remainder of the paper follows closely the general framework of Zerner in \cite{zERWZdstrip} in which he proved a criteria for recurrence/transience of nearest neighbor excited random walks on $\Z^d$ or on a strip $\Z \times \{0,1,\ldots,L-1\} \subset \Z^2$. 
One can map the strip $\Z \times \{0,1,\ldots,L-1\}$ bijectively to $\Z$ by $(x,y) \mapsto Lx+y$, and thus any excited random walk on a strip can be interpreted as an excited random walk on $\Z$ with uniformly bounded jumps. However, since the vertices at the ``top'' and ``bottom'' of the strip are somewhat different than the other vertices in the strip, even if one assumes a spatially i.i.d.\ cookie environment on the strip the corresponding cookie environment on $\Z$ is no longer spatially ergodic. Thus, the results in \cite{zERWZdstrip} do not directly apply to our setup. 
Most (but not all) of the arguments in \cite{zERWZdstrip} can be easily modified, however, to cover excited random walks on $\Z$ with uniformly bounded jumps. Thus, the main novelty of the results in the current paper is that we allow for cookies to induce jump distributions with unbounded support. This introduces serious technical difficulties at several places in Zerner's argument. 
We mention in particular the following. 
\begin{itemize}
 \item The proof that $\P_\eta( \limsup_{n\ra\infty} X_n = \infty) = 1$ follows easily from a martingale argument in \cite{zERWZdstrip}. However, the argument in \cite{zERWZdstrip} is no longer true for excited random walks with possibly unbounded jumps (see Remark \ref{rem:lefttransient} below) and so a new argument is needed. 
 \item The proof of the existence of a stationary distribution for the cookie environment (Section \ref{sec:cep} below) is significantly more difficult than in \cite{zERWZdstrip}. Since we are allowing for unbounded jumps the space of cookie environments is non-compact, and certain tightness results which follow immediately in \cite{zERWZdstrip} must be proved more directly in this paper (see Lemma \ref{lem:tightness} below). 
\end{itemize}
We also mention that even in the case of excited random walks with uniformly bounded jumps, the 0-1 laws in Section \ref{sec:01} are much more difficult than in \cite{zERWZdstrip} since our weak ellipticity condition in Assumption \ref{asmell} is much weaker than the uniform ellipticity assumption in \cite{zERWZdstrip}.

The outline of the paper is as follows. 
In Section \ref{sec:notation} we describe a basic martingale estimate due to Zerner that will be the basis for much of our analysis. 
Theorem \ref{thm:ergrt} is then proved in Section \ref{sec:ergrt}, where we note that a technical difficulty is proving that $\liminf_{n\ra\infty} X_n = -\infty$ and $\limsup_{n\ra\infty} X_n = \infty$ imply that every site is visited infinitely often. 
In Section \ref{sec:cep} we introduce the ``cookie environment process'' which is analagous to the environment viewed from the point of view of the particle in random walks in random environments. The main result in Section \ref{sec:cep} is the existence of a stationary measure for the cookie environment process with certain nice properties. 
In Section \ref{sec:01} we prove a 0-1 law for recurrence/transience under the additional Assumptions \ref{asmiid} and \ref{asmell}, and finally in Section \ref{sec:sharp} we use the cookie environment process and the 0-1 law to prove the sharp criterion for recurrence/transience from Theorem \ref{thm:rt}.

\section{Notation and Preliminary estimates} \label{sec:notation}

We begin by introducing some notation that will be used throughout the paper.  For any $x\in \Z$, let $\tau_x$ and $\s_x$ be the first time the random walk goes above or below $x$, respectively. That is, 
\[
 \tau_x = \inf\{ n\geq 0: \, X_n \geq x \}, \quad \text{and}\quad \s_x = \inf\{ n\geq 0: \, X_n \leq x \}. 
\]
Next, since the transitions of the excited random walk depend on the past history of the walk, we need some notation to keep track of this information. 
Define a function
$\a: \Omega \times \Z_+^\Z \ra \Omega$
by 
\[
 \a(\w,\ell)_{x,j} = \w_{x,\ell(x)+j}, \quad x \in \Z, \, j\geq 1. 
\]
That is, $\a(\w,\ell)$ is the cookie environment $\w$ modified by removing the first $\ell(x)$ cookies from site $x$ for each $x\in \Z$. 
\begin{defn}
 The function $\a$ defined above gives a natural partial ordering on $\Omega_{M,\mu}^+$. We say that $\w' \leq \w$ if there exists an $\ell \in \Z_+^\Z$ such that $\w' = \a(\w,\ell)$.
(Note that the fact that this defines a partial ordering on $\Omega_{M,\mu}^+$ relies on the fact that $\w_{x,j}(\cdot) = \mu(\cdot)$ for all $x \in \Z$, $j> M$.)
\end{defn}

\begin{defn}
 The process $\{\bar\w(n)\}_{n\geq 0}$ is defined by $\bar\w(0) = \w$ and $\bar\w(n) = \a(\w,L_{n-1})$ for $n\geq 1$, where $L_{n-1} = \{L_{n-1}(x)\}_{x\in \Z}$ as defined as in \eqref{ERWdef} is the local time of the walk after $n-1$ steps. 
That is, $\bar\w(n)$ is the cookie environment remaining just before the walk makes the $n$-th step. 
\end{defn}

\subsection{A martingale argument}
The key to the proofs of the main results in the present paper is the adaptation of a martingale argument that was first introduced by Zerner in \cite{zMERW}. 
For $\w \in \Omega_{M,\mu}^+$ and $x\in \Z$, let $\bar{d}(\w_x) = \sum_{j=1}^M d(\w_{x,j})$ be the total drift contained in the cookies at $x$. 
Also, for $x\in \Z$ and $n\geq 1$ let $D_n^x = \sum_{j=1}^{L_{n-1}(x)} d(\w_{x,j})$ be the
the total drift the excited random walk has ``consumed'' at site $x$ in the first $n-1$ steps. 
Then $D_n = \sum_{x\in \Z} D_n^x$ is the total drift consumed by the excited random walk by time $n$.
Then $M_n = X_n - D_n$ is a martingale under the quenched measure $P_\w^x$ with respect to the natural filtration $\mathcal{F}_n = \s(X_k: \, k\leq n)$. 
The optional stopping theorem implies that for any cookie environment $\w \in \Omega_{M,\mu}^+$ and integers $x < y < z$ that 
\begin{equation}\label{EXED}
 E_\w^y \left[ X_{\tau_{z} \wedge \s_{x}} \right] = y + E_\w^y \left[ D_{\tau_{z} \wedge \s_{x}} \right] 
\leq y + \sum_{k = x+1}^{z-1} \bar{d}( \w_k).
\end{equation}
To justify the application of the optional stopping theorem, it is enough to prove that 
\begin{equation}\label{etfinite}
 E_\w^y[ \tau_{z} \wedge \s_{x} ] < \infty
\end{equation}
and 
\begin{equation}\label{jumpsize}
 E_\w^y[|X_{t+1}-X_t| \, | \, \mathcal{F}_t ]  \leq C < \infty, \quad \text{for all }t < \tau_{z} \wedge \s_{x}.
\end{equation}
The expectation in \eqref{etfinite} is finite for any fixed $x,y,z$ and $\w \in \Omega_{M,\mu}^+$ since the exit time of any fixed interval has exponentially decaying tails due to the fact that there are only $M$ cookies at each site. 
The uniform upper bound in \eqref{jumpsize} holds with $C = \max_{k \in (x,z)} \max_{j\leq M} \sum_\ell |\ell| \, \w_{k,j}(\ell) < \infty$. 

In addition to the martingale estimate in \eqref{EXED}, we will also need the following minor modification. 
\begin{lem}\label{lem:mgpos}
 Let $D_n^+ = \sum_{x\geq 0} D_n^x$ be the total drift used to the right of the origin by time $n$. 
If the cookie environment $\w$ is such that 
\begin{itemize}
 \item $\w_x = \{\w_{x,j}\}_{j\geq 1} \in \mathcal{M}_+^M \times \{\mu\}^{\N}$ for all $x \geq 0$
 \item $P_\w^x( \limsup_{n\ra\infty} X_n = \infty ) = 1$ for all $x\geq 0$
\end{itemize}
then 
 \begin{equation}\label{EwDplus}
 E_\w^x[D_{\tau_n}^+] 
\leq n + E_\w^x\left[(X_{\tau_n} - n)\ind{X_{\tau_n-1} \geq 0} \right], \quad \forall x \in [0,n).
\end{equation}
\end{lem}

\begin{rem}
 We note that the conditions on the environment in Lemma \ref{lem:mgpos} are slightly weaker than assuming $\w \in \Omega_{M,\mu}^+$. 
Indeed, the first condition is simply that the cookie environment to the right of the origin is the same as some cookie environment in $\Omega_{M,\mu}^+$. Moreover, it will follow from Proposition \ref{limsupprop} below that the second condition holds for $\eta$-a.e.\ cookie environment $\w \in \Omega_{M,\mu}^+$. 
\end{rem}

\begin{proof}
The proof is accomplished by a modification of the martingale argument leading to \eqref{EXED}. 
For $n\geq 1$ we will modify the random walk process by only recording the times when the random walk is to the right of the origin. That is, any jumps to the left of the origin are replaced by a jump to wherever the walk then jumps back to the right of the origin. However, we make one further modification; if when the walk returns to the right of the origin it lands in $[n,\infty)$ then we make the walk land exactly at $n$. Note that the modification done in this way doesn't change the local time of the sites $x \in [0,n)$ up until time $\tau_n$. Secondly, any jumps to the left of the origin are replaced by a jump somewhere in $[0,n]$ so that the expected displacement only increases. Thus, if we subtract the original drift (that is the expected displacement of the unmodified excited random walk) at all the sites in $[0,n)$ that are visited then we get a sub-martingale. 
To make this precise, let $u_0 = 0$ and let $u_k = \inf \{ i > u_{k-1} : \, X_i \geq 0 \}$ and for $n\geq 1$ fixed let $Y_k^{(n)}$ be defined for $k\geq 1$ by 
\[
 Y_k^{(n)} = 
\begin{cases}
 X_{u_k} & \text{if } u_k < \tau_n \\
 X_{\tau_n} & \text{if } u_k \geq \tau_n \text{ and } X_{\tau_n-1} \geq 0 \\
 n & \text{if } u_k \geq \tau_n \text{ and } X_{\tau_n-1} < 0.
\end{cases}
\]
Then, we claim that $M_k^{(n)} = Y_k^{(n)} - D_{u_k \wedge \tau_n}^+$ is a sub-martingale with respect to the filtration $\mathcal{F}_{u_k} = \s(X_i, \, i\leq u_k)$. 
Note that $M_k^{(n)}$ is constant for $u_k \geq \tau_n$. However, on the event $\{u_k < \tau_n\} \in \mathcal{F}_{u_k}$ we have  
\begin{align*}
E_\w&\left[ M_{k+1}^{(n)} \, | \, \mathcal{F}_{u_k} \right] - M_k^{(n)} \\
&= E_\w\left[ \left( X_{u_{k+1}} - X_{u_k} \right)\ind{u_{k+1} = u_k+1} + \left( X_{u_{k+1}}\wedge n - X_{u_k} \right) \ind{u_{k+1} > u_k + 1} \, | \, \mathcal{F}_{u_k} \right]  \\
&\qquad - E_\w\left[ X_{u_k+1} - X_{u_k}  \, | \, \mathcal{F}_{u_k} \right] \\
&= E_\w\left[ \left( X_{u_{k+1}}\wedge n - X_{u_k+1} \right) \ind{u_{k+1} > u_k + 1} \right] \\
&\geq 0, 
\end{align*}
where the last inequality follows from the fact that on the event $\{u_{k+1} > u_k + 1\}$ we have $X_{u_k+1}< 0 \leq X_{u_{k+1}}$. 
Since $M_k^{(n)}$ is a sub-martingale, \eqref{EwDplus} follows easily from optional stopping. 
\end{proof}

\section{Proof of Theorem \ref{thm:ergrt}}\label{sec:ergrt}


The key to proving part \ref{K01law} of Theorem \ref{thm:ergrt} is the following Proposition. 

\begin{prop}\label{limsupprop}
 Let Assumptions \ref{asmMmu}--\ref{asmic} hold. Then $\P_\eta( \limsup_{n\ra\infty} X_n = \infty ) = 1$. 
\end{prop}

\begin{rem}\label{rem:lefttransient}
Since we have assumed that our cookie environments have a non-negative drift at each step, it is natural to expect that $\limsup_{n\ra\infty} X_n = \infty$. 
However, since we are allowing for possibly unbounded jumps we cannot conclude that this is the case for every cookie environment $\w \in \Omega_{M,\mu}^+$. 
For instance, let $\w= \{\w_{x,j}(\cdot)\}_{x\in \Z, j\geq 1}$ be given by 
\[
 \w_{x,j} =
\begin{cases}
 \frac{1}{2} \d_{-1} + \frac{1}{2} \d_{1} & x \geq -1 \text{ or } j \geq 2\\
 \left( 1 - \frac{1}{x^2} \right) \d_{-1} + \frac{1}{x^2} \d_{x^2-1} & x \leq -2 \text{ and } j= 1. 
\end{cases}
\]
(Here, $\d_x \in M_1(\Z)$ denotes the Dirac measure at $x\in \Z$.)
Note that for this cookie environment $\w$,  $d(\w_{x,j}) = 0$ for all $x\in \Z$, $j\geq 1$ but  
\[
 P_\w( X_n = -n, \, \forall n\geq 0) = \frac{1}{4} \prod_{k=2}^{\infty} \left( 1 - \frac{1}{k^2} \right) = \frac{1}{8} > 0. 
\]
In contrast, if instead we had restricted ourselves to excited random walks with uniformly bounded jumps, then an easy adaptation of the proof of \cite[Lemma 3]{zERWZdstrip} could be used to show that $P_\w( \limsup_{n\ra\infty} X_n = \infty) = 1$ for \emph{all} cookie  environments $\w$
\end{rem}
\begin{proof}
Clearly it is enough to show that $\P_\eta( \tau_m = \infty) = 0$ for all $m\geq 1$, or equivalently
\begin{equation}\label{goright}
 P_\w(\tau_m = \infty) = 0, \quad \forall m\geq 1, \quad \text{for $\eta$-a.e. cookie environment $\w$.}
\end{equation}
For $m\geq 1$ fixed, define the random variables $\{s_k\}_{k\geq 0}$ and $\{S_k\}_{k\geq 0}$ recursively by letting
\[
 s_0 = S_0 = 0, \quad\text{and}\quad 
\begin{cases}
 s_k = \tau_{m} \wedge \s_{2 S_{k-1} - m} \\
 S_k = X_{s_k}, 
\end{cases}
\quad\text{for } k\geq 1. 
\]
Note that these random variables have the following properties. 
\begin{itemize}
 \item If $S_k = X_{s_k} \geq m$ for some $k$ then $s_j = s_k$ for all $j>k$ (and also $S_j = S_k$ for all $j>k$). 
Therefore, $\tau_m = \lim_{k\ra\infty} s_k$ almost surely. 
 \item If $S_k = X_{s_k} < m$, then $s_{k+1}$ is the exit time of the interval $(2S_{k}-m, m)$ and $S_k$ is the midpoint of this interval. 
 \item If $S_k < m$, then $S_k \leq m (1-2^k)$. 
\end{itemize}
Now, these properties imply that
\begin{align}
& P_\w\left( s_k < \tau_{m} \right) = P_\w\left( s_{k-1} < s_k < \tau_{m} \right) \nonumber \\
&= \sum_{n=1}^\infty \sum_{x \leq m(1-2^{k-1})} \sum_{\w' \leq \w} P_\w\left( s_{k-1} = n, \, S_{k-1} = x, \,  \bar\w(n) = \w' \right) P_{\w'}^{x}\left( \s_{2x-m} < \tau_{m} \right). \label{Psksm}
\end{align}
(There are uncountably many cookie environments $\w'$ with $\w' \leq \w$. However, for $n$ and $x$ fixed there are only finitely many $\w'$ for with the first probability inside the sum is non-zero.)
We will show below that
\begin{equation}\label{Pwxuub}
 \lim_{x\ra-\infty} \sup_{\w'\leq \w} P_{\w'}^x \left( \s_{2x-m} < \tau_m \right) \leq \frac{1}{2}, \quad \eta\text{-a.s.} 
\end{equation}
Postponing the proof of \eqref{Pwxuub} for now, we note that this implies that for $\eta$-a.e.\ cookie environment $\w$ there exists a $k_0 = k_0(\w)$ such that 
\[
 P_{\w'}^x\left( \s_{2x-m} < \tau_m \right) \leq \frac{3}{4} , \qquad \forall \w' \leq \w, \, x \leq m(1-2^{k_0-1}). 
\]
Therefore, for any $k\geq k_0$ we can conclude from \eqref{Psksm} that $P_\w( s_k < \tau_m ) \leq \frac{3}{4} P_\w( s_{k-1} < \tau_m )$. 
Iterating this we obtain that $P_\w(s_k < \tau_m) \leq (3/4)^{k-k_0}$ for all $k\geq k_0$, and thus 
\[
 P_\w( \tau_m = \infty) = \lim_{k\ra\infty} P_\w\left( s_k < \tau_m \right) = 0, \quad \eta\text{-a.s.}
\]

To complete the proof of the first part of the Proposition it thus remains to prove \eqref{Pwxuub}. 
To this end, first note that since $\w \in \Omega_{M,\mu}^+$ the random walk is a submartingale under the measure $P_{\w'}^x$ for any $\w'\leq \w$ and therefore
\begin{align*}
 x &\leq E_{\w'}^x\left[ X_{ \tau_m \wedge \s_{2x-m} } \right] \\
&= E_{\w'}^x\left[ X_{\tau_m} \ind{\tau_m < \s_{2x-m}} \right] + E_{\w'}^x\left[ X_{\s_{2x-m}} \ind{\tau_m > \s_{2x-m}} \right] \\
&\leq 2(x-m) P_{\w'}^x( \s_{2x-m} < \tau_m ) + m + E_{\w'}^x\left[ \left( X_{ \tau_m \wedge \s_{2x-m} } - m \right)_+ \right],
\end{align*}
from which it follows that 
\begin{equation}\label{Pwxuub1}
 P_{\w'}^x( \s_{2x-m} < \tau_m ) \leq \frac{1}{2} + \frac{1}{2(m-x)}E_{\w'}^x\left[ \left( X_{ \tau_m \wedge \s_{2x-m} } - m \right)_+ \right]. 
\end{equation}
To bound the last expectation on the right, for any cookie environment $\w$ we can expand the quenched measure $P_\w$ to contain an independent family of random variables $\{W_{y,j}\}_{y \in \Z,\, j\geq 1}$ such that $W_{y,j}$ has distribution $\w_{y,j}$ for each $y \in \Z$ and $j\geq 1$. 
The excited random walk can then be constructed by letting $W_{y,j}$ be the jump that the walk makes upon the $j$-th visit to the site $y$.
Now, suppose that $\w' = \a(\w,\ell) \leq \w$. Then, 
\begin{align}
& E_{\w'}^x\left[ \left( X_{ \tau_m \wedge \s_{2x-m} } - m \right)_+ \right] \nonumber \\
&= E_{\w'}^x\left[ \sum_{y=2x-m+1}^{m-1} \sum_{j\geq 1} \sum_{k=0}^\infty \left( X_{ \tau_m \wedge \s_{2x-m} } - m \right)_+ \ind{X_k=y, \, L_k(y) = j, \, \tau_m \wedge \s_{2x-m} = k+1} \right] \nonumber \\
&\leq   \sum_{y=2x-m+1}^{m-1} \sum_{j=1}^{M-\ell(y)} \sum_{k=0}^\infty E_{\w'}^x\left[ \ind{X_k=y, \, L_k(y) = j, \, \tau_m \wedge \s_{2x-m} > k} \right]E_{\w'}\left[ \left( W_{y,j} + y - m \right)_+ \right]  \nonumber \\ 
&\qquad + B P_{\w'}^x\left( \tau_m \wedge \s_{2x-m} < \infty \right) \nonumber \\
&\leq \sum_{y=2x-m+1}^{m-1} \sum_{j=1}^{M-\ell(y)} E_{\w'}\left[ \left( W_{y,j} + y - m \right)_+ \right] + B \nonumber \\
&\leq \sum_{y=2x-m+1}^{m-1} \sum_{j=1}^{M} E_{\w}\left[ \left( W_{y,j} + y - m \right)_+ \right] + B. \label{rightosub}
\end{align}
where in the first inequality we used the fact that $\w'_{y,j} = \w_{y,\ell(y)+j}= \mu$ has support in $[-B,B]$ for $j + \ell(y) > M$, and in the last inequality we used that $\w' = \a(\w,\ell) \leq \w$.
Applying this upper bound to \eqref{Pwxuub1} we obtain that 
\begin{align*}
 P_{\w'}^x( \s_{2x-m} < \tau_m ) &\leq \frac{1}{2} + \frac{B}{2(m-x)} + \frac{1}{2(m-x)} \sum_{y=2x-m+1}^{m-1} \sum_{j=1}^M E_{\w}\left[ (W_{y,j} + y-m)_+ \right] .
\end{align*}
Note that this upper bound is uniform over all $\w' \leq \w$. 
Therefore, to complete the proof of \eqref{Pwxuub} we need only to show that 
\begin{equation}\label{EWpsum}
 \lim_{x\ra -\infty} \frac{1}{2(m-x)}\sum_{y=2x-m+1}^{m-1} \sum_{j=1}^M E_{\w}\left[ (W_{y,j} + y-m)_+ \right] = 0, \quad \eta\text{-a.s.}
\end{equation}
To this end, note that for any $N\geq 2$ 
\begin{align*}
& \limsup_{x\ra - \infty} \frac{1}{2(m-x)}\sum_{y=2x-m+1}^{m-1} \sum_{j=1}^M E_{\w}\left[ (W_{y,j} + y-m)_+ \right] \\
&\leq \lim_{x\ra - \infty} \frac{1}{2(m-x)} \left\{ \sum_{y=2x-m+1}^{m-N} \sum_{j=1}^M E_{\w}\left[ (W_{y,j} -N)_+ \right] +\sum_{y=m-N+1}^{m-1} \sum_{j=1}^M E_{\w}\left[ |W_{y,j}| \right] \right\} \\
&= \sum_{j=1}^M \E_\eta \left[ (W_{0,j} - N)_+ \right], \qquad \eta\text{-a.s.},
\end{align*}
where the last equality follows from Birkhoff's Ergodic Theorem. 
Since $\d = \sum_{j=1}^M \E_\eta[W_{0,j}] < \infty$, this upper bound can be made arbitrarily small by taking $N\ra\infty$. 
\end{proof}

\begin{lem}\label{infsupR}
Let Assumptions \ref{asmMmu}--\ref{asmic} hold. 
Then 
$$\P_\eta\left( \liminf_{n\ra\infty} X_n = -\infty, \, \limsup_{n\ra\infty} X_n = \infty \right) 
= \P_\eta(\mathcal{R}).$$
\end{lem}
\begin{proof}
Recall that $B$ is the uniform bound on the jumps from distribution $\mu$ as in Assumption \ref{asmmu}. 
Consider the number of times that the random walk makes a jump from some site $y \geq B$ to the left of the origin. By Assumption \ref{asmMmu} and the fact that $\w \in \Omega_{M,\mu}^+$ we can conclude that any such jump must have occured during the first $M$ visits to $y$. Therefore, 
\begin{align*}
 \E_\eta\left[ \sum_{n\geq 0} \ind{X_n \geq B, \, X_{n+1} < 0} \right] 
&= \sum_{n\geq 0}  \sum_{j\leq M} \sum_{y\geq B} \P_\eta( X_n = y, \, L_{n-1}(y) = j, \, X_{n+1} < 0 ) \\
&\leq \sum_{j\leq M} \sum_{y\geq B} \P_\eta( W_{y,j} + y < 0 ) \\
&= \sum_{j\leq M} \sum_{y \geq B} \P_\eta( W_{0,j} + y < 0)\\
&= \sum_{j\leq M} \E_\eta\left[ ( W_{0,j} + B)_- \right] < \infty,
\end{align*}
where the last line is finite by Assumption \ref{asmic}. 
Thus, if the random walk crosses the interval $[0,B)$ infinitely many times then it must in fact visit that interval infinitely often. That is, 
\[
 \P_\eta\left( \liminf_{n\ra\infty} X_n = -\infty, \, \limsup_{n\ra\infty} X_n = \infty \right) \leq \P_\eta\left( \sum_{n\geq 0} \ind{X_n \in [0,B) } = \infty \right). 
\]
For any site $y \in [0,B)$ and $x \in \Z$ there is a path of finite length starting at $y$ and ending at $x$ that occurs with positive probability under the random walk measure induced by $\mu$. If the site $y$ is visited infinitely often, then this finite path must be followed upon infinitely many times. 
Thus, if the interval $[0,B)$ is visited infinitely many times then every site is visited infinitely many times. Therefore, 
\[
  \P_\eta\left( \liminf_{n\ra\infty} X_n = -\infty, \, \limsup_{n\ra\infty} X_n = \infty \right) \leq \P_\eta\left( \sum_{n\geq 0} \ind{X_n = x } = \infty , \, \forall x \in \Z \right). 
\]
The reverse inequality is obvious. 
\end{proof}

We are now ready to give the proof of the first main result of the paper.

\begin{proof}[Proof of Theorem \ref{thm:ergrt}]
If any site $x \in \Z$ is visited infinitely many times then it follows from Assumption \ref{asmMmu} that  $\liminf_{n\ra\infty} X_n = -\infty$ and $\limsup_{n\ra\infty} X_n = \infty$. 
Thus, it follows from Proposition \ref{limsupprop} and Lemma \ref{infsupR} that $\P_\eta(\mathcal{R} \cup \mathcal{T}_+) = 1$. This proves the first part of Theorem \ref{thm:ergrt}.

To prove the second part of Theorem \ref{thm:ergrt}, it is enough to show that $\d<1$ implies
\begin{equation}\label{sureleft}
 P_\w( \s_{-m} = \infty) = 0, \quad \forall m\geq 1, \quad \text{for $\eta$-a.e.\ cookie environment $\w$.}
\end{equation}
The proof of \eqref{sureleft} is similar to the proof of \eqref{goright} above and so we will only provide an outline here highlighting the differences from the proof of \eqref{goright}. 
To this end, fix an integer $a > \frac{\d}{1-\d}$ and for $m\geq 1$ fixed define the random variables $\{r_k\}_{k\geq 0}$ and $\{R_k\}_{k\geq 0}$ by 
\[
r_0=R_0 = 0, \quad\text{and}\quad
\begin{cases}
 r_k = \tau_{(a+1)R_{k-1} + am} \wedge \s_{-m} \\
R_k = X_{r_k}
\end{cases}
\quad \text{ for } k\geq 1. 
\]
These random variables play a similar role as $s_k$ and $S_k$ in the proof of \eqref{goright} above, but now if $R_k>-m$ then $R_k$ is no longer the midpoint of the interval $(-m,(a+1)R_k+am)$ but instead the distance to the right endpoint is $a$ times the distance to the left endpoint. 
As above, the key to the proof of \eqref{sureleft} will be showing that 
\begin{equation}\label{limexitub}
 \lim_{x\ra\infty} \sup_{\w' \leq \w} P_{\w'}^{x}\left( \tau_{(a+1)x + am} < \s_{-m} \right) \leq \frac{1}{1+a} + \d, \quad \eta\text{-a.s.}
\end{equation}
Note that our choice of $a$ implies that $1/(1+a) + \d < 1$. 
Therefore, \eqref{limexitub} implies that any $\rho \in ( \frac{1}{1+a} + \d, 1)$ and $\eta$-a.e.\ cookie environment $\w$
\begin{align*}
& P_\w\left( r_k < \s_{-m} \right) = P_\w\left( r_{k-1} < r_k < \s_{-m} \right) \nonumber \\
&= \sum_{n=1}^\infty \sum_{x\geq m(1+a)^{k-1} - m} \sum_{\w' \leq \w} P_\w\left( r_{k-1} = n, \, R_{k-1} = x, \,  \bar\w(n) = \w' \right) P_{\w'}^{x}\left( \tau_{(a+1)x + am} < \s_{-m} \right),  \label{Prksm} \\
&\leq \rho P_\w( r_{k-1} < \s_{-m} ),
\end{align*}
for all $k$ sufficiently large (depending on $\w$). 
From this it follows that  $$P_\w( \s_{-m} = \infty) = \lim_{k\ra\infty} P_\w( r_k < \s_{-m} ) = 0, \qquad \eta\text{-a.s.}$$

To complete the proof of \eqref{sureleft} we need to prove \eqref{limexitub}. 
As a first step in this direction, note that \eqref{EXED} implies that 
\begin{equation}\label{rcexitub}
 E_{\w'}^x\left[ X_{\tau_{(a+1)x+am} \wedge \s_{-m}} \right] \leq x + \sum_{i=-m+1}^{(a+1)x+am-1} \bar{d}(\w'_i) \leq x + \sum_{i=-m+1}^{(a+1)x+am-1} \bar{d}(\w_i), 
\end{equation}
and a lower bound for the same expectation is 
\begin{align}
& E_{\w'}^x\left[ X_{\tau_{(a+1)x+am} \wedge \s_{-m}} \right] \nonumber \\
&\geq (a+1)(x+m) P_{\w'}^x\left( \tau_{(a+1)x+am} < \s_{-m} \right) - m - E_{\w'}^x\left[ (X_{\tau_{(a+1)x+am} \wedge \s_{-m}}+m)_-  \right] \label{rcexitlb}
\end{align}
Combining \eqref{rcexitub} and \eqref{rcexitlb} we get 
\begin{align}
 P_{\w'}^x\left( \tau_{(a+1)x+am} < \s_{-m} \right) & \leq \frac{1}{a+1} + \frac{1}{(a+1)(x+m)} \sum_{i=-m+1}^{(a+1)x+am-1} \bar{d}(\w_i) \nonumber  \\
&\qquad +  \frac{1}{(a+1)(x+m)} E_{\w'}^x\left[ (X_{\tau_{(a+1)x+am} \wedge \s_{-m}}+m)_-  \right]. \label{roughub}
\end{align}
Since the ergodic theorem implies that the second term on the right converges to $E_\eta[ \bar{d}(\w_0) ] = \d$, $\eta$-a.s., to complete the proof of \eqref{limexitub} we need only to show that 
\begin{equation}\label{undershoot}
 \lim_{x\ra\infty} \sup_{\w'\leq \w} \frac{1}{(a+1)(x+m)} E_{\w'}^x\left[ (X_{\tau_{(a+1)x+am} \wedge \s_{-m}}+m)_-  \right] = 0. 
\end{equation}
To this end, an argument similar to \eqref{rightosub} implies that 
\begin{align*}
E_{\w'}^x\left[ (X_{\tau_{(a+1)x+am} \wedge \s_{-m}}+m)_-  \right] \leq \sum_{y=-m+1}^{(a+1)x+am-1} \sum_{j=1}^M E_{\w}\left[ (W_{y,j} +y + m)_- \right]  + B, \qquad \forall \w' \leq \w,
\end{align*}
and from this \eqref{undershoot} follows from an argument similar to the proof of \eqref{EWpsum}.
\end{proof}

\section{The cookie environment process}\label{sec:cep}

Since the transitions of the excited random walk depend on the history of the walk, the excited random walk is not Markovian. Of course, by expanding the state space to include some information about the history of the walk one can obtain a Markov process. In this section we will introduce a Markov chain which we will call the \emph{cookie environment process} that is reminiscent of the environment viewed from the particle that is used in random walks in random environments. This cookie environment process is similar to a process used by Zerner in \cite{zERWZdstrip} in studying excited random walks on $\Z^d$ and on strips, and the main outline of this section follows that \cite{zERWZdstrip} though the possibility of unbounded jumps introduces significant difficulties.

To define the cookie environment process, let $\theta:\Omega \ra \Omega$ be the natural (spatial) left shift operator on cookie environments. That is, $\th \w = \w' = \{\w'_{x,j}\}_{x\in\Z, \, j\geq 1}$ is the cookie environment with $\w'_{x,j} = \w_{x+1,j}$ for all $x \in \Z$, $j\geq 1$. 
We define the cookie environment process $\{\zeta_n\}_{n\geq 0}$ to be the stochastic process on $\Omega \times \Z_+$ defined by 
\begin{equation}\label{Cndef}
 \zeta_n = \left( \theta^n \bar{\w}(\tau_n), \, X_{\tau_n} - n \right), \quad \forall n\geq 0. 
\end{equation}
The first coordinate $\theta^n \bar{\w}(\tau_n)$ records the cookie environment when the walk first enters $[n,\infty)$ and shifts the cookie environment $n$ units to the left. The second coordinate $X_{\tau_n} - n$ records the distance which the excited random walk exceeds $n$ when first entering $[n,\infty)$. 
Note that the cookie environment process $\zeta_n$ is only well defined for all $n$ when $\limsup_{n\ra\infty} X_n = +\infty$. 
Therefore, if we define $\Psi \subset \Omega \times \Z_+$ by 
\[
 \Psi = \left\{ (\w,x) \in\Omega \times \Z_+: \, P_\w^x\left( \limsup_{n\ra\infty} X_n = \infty \right) = 1 \right\}, 
\]
it is easy to see that $\{\zeta_n\}_{n\geq 0}$ is a Markov chain on $\Psi$.
In the following lemma, we will show that $\zeta_n$ is in fact a weak Feller continuous Markov chain on $\Psi$. However, before stating this a brief discussion is needed on the topology of the state space $\Psi$. 
The space $M_1(\Z)$ of probability distributions on $\Z$ is equipped with the topology of weak-* convergence (convergence in distribution), and then the space $\Omega = M_1(\Z)^{\Z \times \Z_+}$ of all cookie environments is given the corresponding product topology. Note that it is well known that this topology on $\Omega$ is compatible with a metric under which $\Omega$ is a Polish space. 
Finally, the space $\Psi$ is given the induced subspace topology from $\Omega \times \Z_+$, where $\Z_+$ is of course given the discrete topology. 

\begin{lem}\label{lem:weakFeller}
 The cookie environment process $\{\zeta_n\}_{n\geq 0}$ is a weak Feller continuous Markov chain on $\Psi$. That is, if $f: \Psi \ra \R$ is a bounded continuous function 
then the mapping $(\w,x) \mapsto E_\w^x[ f(\zeta_1) ]$ is a bounded continuous function from $\Psi \ra \R$. 
\end{lem}

\begin{rem}
 The proof below is a minor adaptation of the proof of Lemma 5 in \cite{zERWZdstrip}. 
\end{rem}

\begin{proof}
Let $(\w,x) \in \Psi$ be fixed and let $(\w^{(n)},x_n)$ be a sequence in $\Psi$ with $(\w^{(n)},x_n) \ra (\w,x)$. 
Since $x_n \ra x$ in the discrete topology on $\Z_+$, without loss of generality we may assume that $x_n = x$ for all $n$. 
If $x\geq 1$ then $\tau_1 = 0$ and so 
\[
 E_{\w^{(n)}}^x[ f(\zeta_1) ] = f(\th \w^{(n)},x-1) \xrightarrow[n\ra\infty]{} f(\th \w, x-1) = E_\w^x[ f(\zeta_1) ], \quad \forall x\geq 1. 
\]
Thus, to complete the proof we need only to handle the case when $x=0$; that is, we need to show that 
\[
 \lim_{n\ra\infty} E_{\w^{(n)}}[ f(\zeta_1) ] = E_\w[ f(\zeta_1) ]. 
\]
Since $(\w,0) \in \Psi$, we have that $\tau_1$ is finite $P_\w$-a.s., and so for any $\e>0$ there exists a $t = t(\w,\e)$ and $L=L(\w,\e)$ such that 
$P_\w(\tau_1 \leq t, \, \sup_{k\leq \tau_1} |X_k| \leq L) >  1-\e$. 
Let $\Pi_{t,L}$ denote the set of paths $\mathbf{x} = (x_0,x_1,\ldots,x_m)$ of some length $m\leq t$ such that 
$x_0 = 0$, $-L \leq x_i \leq 0$ for all $i<m$, and $1 \leq x_m \leq L$.
With this notation we have that 
\begin{equation}\label{Ptau1tL}
 P_\w\left( \tau_1 \leq t, \, \sup_{k\leq \tau_1} |X_k| \leq L \right) = \sum_{\mathbf{x} \in \Pi_{t,L}} P_\w\left( X_\cdot \text{ follows the path } \mathbf{x} \right).  
\end{equation}
For each fixed path $\mathbf{x} \in \Pi_{t,L}$, the probability inside the sum on the right depends on the cookie environment $\w$ only through $\{\w_{x,j}(z)\}_{|x|\leq L, \, |z|\leq 2L, \, j\leq M}$ and the fixed probability measure $\mu$. 
Since we have given the space of cookie environments $\Omega$ the product topology, it follows that
\begin{equation}\label{PwPwn}
 \lim_{\w^{(n)} \ra \w} P_{\w^{(n)}}\left( X_\cdot \text{ follows the path } \mathbf{x} \right) = P_\w\left( X_\cdot \text{ follows the path } \mathbf{x} \right), \quad \forall \mathbf{x} \in \Pi_{t,L}. 
\end{equation}
Since the sum in \eqref{Ptau1tL} is over the finite set $\Pi_{t,L}$, we can thus conclude by the choice of $t$ and $L$ that $P_{\w^{(n)}}(\tau_1 \leq t, \, \sup_{k\leq \tau_1} |X_k| \leq L) \geq  1-\e$ for all sufficiently large $n$. 
For any $z =1,2,\ldots,L$ let $\Pi_{t,L}^z \subset \Pi_{t,L}$ be the subset of paths that ends at $z$. 
Also, for any $\mathbf{x} = (x_0,x_1,\ldots,x_m) \in \Pi_{t,L}$ let $\ell_{\mathbf{x}}: \Z \ra \Z_+$ be given by $\ell_{\mathbf{x}}(y) = \sum_{k=0}^{m-1} \ind{x_k = y}$. That is, $\ell_{\mathbf{x}}$ is the local time of the random walk following the path $\mathbf{x}$.  
With this notation, if $\mathbf{x} \in \Pi_{t,L}^z$ and the walk $X_\cdot$ follows the path $\mathbf{x}$ to begin then $\zeta_1 = (\theta \, \a(\w,\ell_{\mathbf{x}}), z-1)$. 
Therefore, 
\begin{align*}
& \left| E_{\w^{(n)}}[ f(\zeta_1) ] - E_\w[ f(\zeta_1) ] \right| \\
&\leq \left| E_{\w^{(n)}}[ f(\zeta_1) \ind{\tau_1 \leq t, \, \sup_{k\leq \tau_1} |X_k| \leq L} ] - E_\w[ f(\zeta_1) \ind{\tau_1 \leq t, \, \sup_{k\leq \tau_1} |X_k| \leq L} ] \right| + 2 \e \|f\|_\infty \\
&\leq \sum_{z=1}^L \sum_{\mathbf{x} \in \Pi_{t,L}^z} \biggl| P_{\w^{(n)}}( X_\cdot \text{ follows the path } \mathbf{x} ) f(\th \, \a(\w^{(n)},\ell_{\mathbf{x}}),z-1)   \\
&\qquad \qquad \qquad - P_\w( X_\cdot \text{ follows the path } \mathbf{x} ) f( \th \, \a(\w,\ell_{\mathbf{x}}),z-1) \biggr| + 2 \e \|f\|_\infty. 
\end{align*}
The terms in the summation vanish as $n\ra\infty$ due to \eqref{PwPwn} and the fact that 
$\w^{(n)}\ra \w$ implies that $\th \, \a(\w^{(n)},\ell_{\mathbf{x}}) \ra \th \, \a(\w,\ell_{\mathbf{x}})$. 
Therefore, as $n\ra\infty$ the above is at most $2\e \|f\|_\infty$. Since $\e>0$ was arbitrary this completes the proof of the lemma. 
\end{proof}

Our main goal in this section is to establish the existence of a stationary measure $\pi$ for the cookie environment process. Since we have shown that the $\zeta_n$ is weak Feller continuous on $\Psi$, a natural approach is to obtain $\pi$ as a (subsequential) limit of the sequence of 
measures $\mu_n$ on $\Psi$ defined by 
\begin{equation}\label{mundef}
 \mu_n(\cdot) = \frac{1}{n} \sum_{k=1}^n \P_\eta( \zeta_k \in \cdot ). 
\end{equation}
To this end, we will need the following tightness result.

\begin{lem}\label{lem:tightness}
If Assumptions \ref{asmMmu}--\ref{asmic} hold, then the sequence $\mu_n$ defined in \eqref{mundef} is a tight sequence of measures on $\Psi$. 
\end{lem}
\begin{rem}
 If we restrict ourselves to cookie environments with uniformly bounded jumps, then the space of all cookie environments is compact and the tightness of the sequence $\{\mu_n\}$ is immediate. 
\end{rem}

\begin{proof}
We will break up the proof into three preliminary steps: 1) Finding ``nice'' compact subsets of $\Omega$, 2) Finding ``nice'' closed subsets of $\Psi$, and 3) Controlling the second coordinate of $\zeta_k$,  $X_{\tau_k} - k$. 

\noindent\textbf{Step 1: Nice compact subsets of $\Omega = M_1(\Z)^{\Z \times \Z_+}$.}\\
We claim that for any $\e>0$ there exists a compact $\mathcal{K}_\e \subset \Omega$ such that 
\begin{equation}\label{Keprop}
 \eta( \w \in \mathcal{K}_\e ) \geq 1- 2\e
\qquad\text{and}\qquad
 \w \in \mathcal{K}_\e \Longrightarrow \w' \in \mathcal{K}_\e, \, \forall \w'\leq \w. 
\end{equation}
To prove this claim, first note that for any $\d>0$ there exists a compact subset $\mathcal{C}_\delta \subset M_1(\Z)$ such that $\eta( \w_0 \in \mathcal{C}_\delta^\N ) \geq 1-\d$. 
This is possible since we have that $\w_{0,j} = \mu$ for all $j > M$ and thus we need only to choose a compact $\mathcal{C}_\delta$ such that $\mu \in \mathcal{C}_\delta$ and $\eta( \w_{0,j} \in \mathcal{C}_\d ) \geq 1-\d/M$ for $j=1,2,\ldots,M$ (this can be accomplished since $M_1(\Z)$ is a Polish space and every probability distribution on a Polish space is tight.)
Now let 
\[
 \mathcal{K}_\e = \bigotimes_{x\in \Z} \mathcal{C}_{\e 3^{-|x|}}^\N = \left\{\w \in \Omega: \, \w_{x,j} \in \mathcal{C}_{\e 3^{-|x|}}, \, \forall x\in \Z, \, j\geq 1 \right\}. 
\]
Since $\mathcal{K}_\e$ is the product of compact sets, it follows from Tychonoff's Theorem that $\mathcal{K}_\e$ is a compact subset of $\Omega_{M,\mu}^+$. 
It is clear from the construction of $\mathcal{K}_\e$ that $\w \in \mathcal{K}_\e$ implies that $\w' \in \mathcal{K}_\e$ for all $\w' \leq \w$, and also $ \eta( \w \notin \mathcal{K}_\e) \leq \sum_{x \in \Z} \e 3^{-|x|} = 2 \e$. 
This completes the proof of the claim in \eqref{Keprop}. 

\noindent\textbf{Step 2: Nice closed subset of $\Psi$.}\\
A drawback to the compact subset $\mathcal{K}_\e$ constructed in step 1 above is that we cannot conclude that 
$\mathcal{K}_\e \times \Z_+ \subset \Psi$. 
Moreover, 
since $\Psi$ is not a closed subset of $\Omega \times \Z_+$ 
we cannot conclude that $(\mathcal{K}_\e \times [0,L]) \cap \Psi$ is a relatively compact subset of $\Psi$. 
Therefore, we will show in this step that for any $\e>0$ there exists a closed subset $F_\e \subset \Omega$ with the following three properties. 
\begin{equation}\label{Feprop}
 \eta( \w \in F_\e) \geq 1-\e, 
\qquad
\w \in F_\e \Longrightarrow \w' \in F_\e, \, \forall \w'\leq \w,
\quad\text{and}\quad
F_\e \times \Z_+ \subset \Psi. 
\end{equation}
First of all, it is easy to see that for any $u < x < z$ that  
\begin{equation}\label{Pexitclosed}
 \left\{\w \in \Omega: \, P_\w^x( \s_u < \tau_z ) \leq \frac{3}{4} \right\}
\end{equation}
is a relatively closed subset of $\Omega$. 
Indeed, if $\w \in \Omega$ is such that $P_\w^x( \s_u < \tau_z ) > \frac{3}{4}$ then there exists a $t<\infty$ and $L<\infty$ (depending on $\w,u,x$ and $z$) such that 
\begin{equation}\label{PexittL}
 P_\w^x\left( \s_u < (\tau_z \wedge t), \, |X_{\s_u}|\leq L \right) > \frac{3}{4}. 
\end{equation}
Since the event inside the probability on the left concerns only finitely many possible paths, the probability on the left is easily seen to be a continuous function of finitely many of the coordinates $\w_{y,j}(z)$ of the cookie environment $\w$. Thus, there is a neighborhood of $\w \in \Omega$ for which the inequality \eqref{PexittL} holds for all cookie environments in this neighborhood, and this is enough to show that the set in \eqref{Pexitclosed} is closed. 
Since \eqref{Pexitclosed} is closed it follows for $k,m\geq 1$ that 
\begin{align*}
A_{k,m} &:= \left\{ \w \in \Omega: P_{\w'}^x \left( \s_{2x-m} < \tau_m \right) \leq \frac{3}{4}, \forall \w'\leq \w, \, \forall x \leq -k \right\} \\
&= \bigcap_{x\leq -k} \bigcap_{\ell \in \Z_+^\Z} \left\{ \w \in \Omega: \, P_{\a(\w,\ell)}^x (\s_{2x-m} < \tau_m ) \leq \frac{3}{4} \right\}
\end{align*}
is a closed subset of $\Omega$. 
(Note we are also using that the mapping $\w \mapsto \a(\w,\ell)$ is continuous for a fixed function $\ell: \Z \ra \Z_+$.)
It follows from \eqref{Pwxuub} that $\lim_{k \ra \infty} \eta( \w \in A_{k,m} ) =1$, and the paragraph following \eqref{Pwxuub} shows that $P_{\w}(\tau_m < \infty) = 1$ for all $\w \in A_{k,m}$. 
Also, it follows clearly from the definition of $A_{k,m}$ above that $\w \in A_{k,m}$ implies that $\w' \in A_{k,m}$ for all $\w'\leq \w$. 
Therefore, if $k(\e,m)<\infty$ is large enough so that $\eta( A_{k(\e,m),m} ) \geq 1 - \frac{\e}{2^m}$  then 
$
 H_\e = \bigcap_{m=1}^\infty A_{k(\e,m),m}, 
$
is a closed subset of $\Omega$ such that 
\[
 \eta(\w \in H_\e) \geq 1-\e
\qquad
\w \in H_\e \Longrightarrow \w' \in H_\e, \, \forall \w'\leq \w,
\quad\text{and}\quad H_\e \times \{0\} \subset \Psi. 
\]
Finally, if $\theta^{-x} H_\e = \{\w: \th^x \w \in H_\e \}$ then $F_\e = \bigcap_{x=0}^\infty \left( \theta^{-x}H_{\e 2^{-x-1}} \right)$ is a closed subset of $\Omega$ with all the required properties. 

\noindent\textbf{Step 3: Control on $X_{\tau_k} - k$.} \\
To obtain control on the second coordinate of the cookie environment process $\zeta_n$, we note that for any integer $L>B$ and $k\geq 1$ that 
\begin{align*}
 \P_\eta( X_{\tau_k} - k \geq L )
&\leq \sum_{y<k} \sum_{j=1}^M \P_\eta( W_{y,j} + y - k \geq L )\\
&= \sum_{j=1}^M \sum_{\ell = L+1}^\infty \P_\eta( W_{0,j} \geq \ell ) 
= \sum_{j=1}^M \E_\eta\left[(W_{0,j}-L)_+ \right].
\end{align*}
Note that this upper bound does not depend on $k\geq 1$ and can be made arbitrarily small by taking $L$ sufficiently large. Thus, for any $\e>0$ there exists an $L_\e < \infty$ such that 
\begin{equation}
 \P_\eta( X_{\tau_k} - k \geq L_\e ) \leq \e, \quad \forall k\geq 1. 
\end{equation}

Having completed the above three preliminary steps, we are now ready to prove the tightness claimed in the statement of the lemma. 
For $\e>0$, let $\mathcal{K}_\e$, $F_\e$ and $L_\e$ be as in steps 1--3 above. Then, $(\mathcal{K}_{\e} \cap F_\e) \times [0,L_\e]$ is a compact subset of $\Psi$, and 
\begin{align*}
 \liminf_{n\ra\infty} & \mu_{n}\left( (\mathcal{K}_{\e} \cap F_\e) \times [0,L_\e] \right) 
= \liminf_{n\ra\infty} \frac{1}{n} \sum_{k=1}^n \P_\eta\left( \zeta_k \in  (\mathcal{K}_{\e} \cap F_\e) \times [0,L_\e]  \right) \\
&\geq 1 - \limsup_{n\ra\infty} \frac{1}{n} \sum_{k=1}^n \left\{ \P_\eta\left( \th^k \bar{\w}(\tau_k) \notin \mathcal{K}_{\e} \cap F_\e \right) + \P_\eta\left( X_{\tau_k} - k > L_\e \right) \right\}  \\
&\geq 1 - \e -  \limsup_{n\ra\infty} \frac{1}{n} \sum_{k=1}^n \eta\left( \th^k \w \notin \mathcal{K}_{\e} \cap F_\e \right) \\
&\geq 1-  \e  - \left\{ \eta( \w \notin \mathcal{K}_\e ) + \eta\left( \w \notin F_\e \right) \right\} \\
&\geq 1- 4 \e.
\end{align*}
where the second inequality follows from the properties of $\mathcal{K}_\e$ and $F_\e$ and the fact that $\th^k \bar{\w}(\tau_k) \leq \th^k \w$. 
Since $\e>0$ was arbitrary, this completes the proof of the lemma. 
\end{proof}

\begin{cor}\label{cor:invariant}
 If Assumptions \ref{asmMmu}--\ref{asmiid} hold 
then there exists a stationary probability measure $\pi$ on $\Psi$ for the Markov chain $\zeta_n$ with the following properties. 
\begin{enumerate}
 \item The environment to the right of the origin $\{\w_x \}_{x \geq 0}$ has the same distribution under $\pi$ as under the measure $\eta$. 
Moreover, if $(\w,X_0)$ has distribution $\pi$ then $\{\w_x\}_{x\geq 0}$ is independent of $X_0$ and the environment to the left of the origin $\{\w_x \}_{x < 0}$. 
\label{piprop1}
 \item $E_\pi[ \bar{d}(\w_y) ] \leq \d$ for all $y \in \Z$. 
\end{enumerate}
\end{cor}

\begin{proof}
It follows from Lemma \ref{lem:tightness} that the sequence of measures $\mu_{n}$ is tight, and thus there exists a subsequence $n_m$ such that $\mu_{n_m}$ converges to a measure $\pi$ on $\Psi$.
The fact that $\pi$ is necessarily a stationary distribution for the Markov chain $\zeta_n$ follows from the fact that $\zeta_n$ is weakly Feller (see the proof of Theorem 12.0.1(i) in  \cite{mtMCSS}). 

To prove the claimed properties of the stationary measure $\pi$, let $A \subset \Omega$ be a measurable with respect to $\s(\w_x, \, x< 0)$, $B \subset \Omega$ be measurable with respect to $\s(\w_x, \, x\geq 0)$, and $y \in \Z_+$. Then 
\begin{align*}
 \mu_n( (A \cap B) \times \{y\})
&= \frac{1}{n} \sum_{k=1}^n \P_\eta \left( \theta^k \bar{\w}(\tau_k) \in A \cap B, \, X_{\tau_k} = k+y \right) \\
&= \frac{1}{n} \sum_{k=1}^n \P_\eta \left( \theta^k \bar{\w}(\tau_k) \in A, \, X_{\tau_k} = k+y \right) \eta( \th^k \w \in B ) \\
&= \mu_n( A \times \{y\})\eta( \w \in B ),
\end{align*}
where the second equality follows from the i.i.d.\ assumption for the law $\eta$ on cookie environments. 
It follows from this that the claimed properties in \ref{piprop1} hold. 

For the second claimed property of $\pi$, note that
\[
 E_{\mu_n} [ \bar{d}(\w_y)]  = \frac{1}{n} \sum_{k=1}^n \E_\eta\left[ \bar{d}(\w_{k+y}) - D_{\tau_k}^{k+y} \right] 
\leq  \frac{1}{n} \sum_{k=1}^n \E_\eta[ \bar{d}(\w_{k+y}) ] 
= \d.
\]
Therefore, to conclude that $E_\pi[\bar{d}(\w_y)] \leq \d$ we need only to show uniform integrability of $\bar{d}(\w_y)$ under the sequence of measures $\mu_n$. For any $L<\infty$
\begin{align*}
  E_{\mu_n} [ \bar{d}(\w_y) \ind{\bar{d}(\w_y) \geq L} ]  
&= \frac{1}{n} \sum_{k=1}^n \E_\eta\left[ \left( \bar{d}(\w_{k+y}) - D_{\tau_k}^{k+y} \right) \ind{ \bar{d}(\w_{k+y}) - D_{\tau_k}^{k+y} \geq L} \right] \\
&\leq  \frac{1}{n} \sum_{k=1}^n \E_\eta[ \bar{d}(\w_{k+y}) \ind{ \bar{d}(\w_{k+y}) \geq L}  ]
= \E_\eta[ \bar{d}(\w_{0}) \ind{ \bar{d}(\w_{0}) \geq L}  ].
\end{align*}
Note that this upper bound is uniform over $n$ and can be made arbitrarily small by taking $L\ra\infty$. This proves the required uniform integrability. 
\end{proof}


For the remainder of the paper, we will often consider an excited random walk where both the cookie environment $\w$ and the starting location $X_0$ are random. If the joint distribution of $(\w,X_0)$ is given by some measure $\a$ on $\Omega \times \Z$ then in a slight abuse of notation we will use $\P_\a(\cdot) = E_\a[ P_\w^{X_0}(\cdot) ]$ to denote the averaged distribution of this walk. Corresponding expectations will be denoted $\E_\a$. 
In particular, $\P_\pi$ will denote the law of the excited random walk when $(\w,X_0)$ are chosen from the stationary measure $\pi$ from Corollary \ref{cor:invariant}. 
We close this section with a result concerning the behavior of the excited random walk under this measure.

\begin{lem}\label{EpiDinf}
 If $\pi$ is an invariant probability measure for $\zeta_n$ as given in Corollary \ref{cor:invariant}, then $\E_\pi[ D_\infty^0] \leq 1$, where 
$D_\infty^0 = \lim_{n\ra\infty} D_n^0$ is the total drift ever consumed by the walk at the origin.  
\end{lem}
\begin{proof}
Recalling \eqref{EwDplus}, for any $(\w,x) \in \Psi$ and $0\leq x<n$ we obtain that 
\begin{align*}
 E_\w^x[ D_{\tau_n}^+ ] &\leq n + E_\w^x\left[ (X_{\tau_n}-n) \ind{X_{\tau_n-1}\geq 0} \right]\\
&\leq n + B + \sum_{j=1}^M \sum_{y=0}^{n-1} E_\w\left[ (W_{y,j} + y-n)_+ \right], 
\end{align*}
where the last inequality follows from an argument similar to that of \eqref{rightosub}. 
Then, averaging both sides of the above inequality with respect to the measure $\pi$ on $(\w,x)$ we obtain that 
\begin{align*}
 \E_\pi[ D_{\tau_n}^+ ] 
&\leq n + B + \sum_{j=1}^M \sum_{y=0}^{n-1} \E_\eta \left[ (W_{y,j} + y-n)_+ \right] \nonumber \\
&= n + B + \sum_{j=1}^M \sum_{k=1}^n \E_\eta\left[ (W_{0,j} - k)_+ \right] \nonumber \\
&= n + o(n), \label{EpiOS}
\end{align*}
where in the first inequality we used also that the environment has the same distribution to the right of the origin under $\pi$ as under the original measure $\eta$, and the last equality follows from the fact that $\E_\eta[ (W_{0,j} - k)_+ ] \ra 0$ as $k\ra \infty$. 

Next, for $1\leq k \leq n$ we have that $D_{\tau_n}^+ \geq \sum_{x=0}^{n-k} D_{\tau_{x+k}}^x$. 
Since $\pi$ is a stationary measure for the cookie environment process it follows that we have that $\E_\pi[ D_{\tau_{x+k}}^x ] = \E_\pi[ D_{\tau_k}^0 ]$ for all $x \geq 0$ and $k\geq 1$, and thus
\begin{equation}\label{Dtaupluslb}
(n-k+1) \E_\pi[ D_{\tau_k}^0 ] \leq \E_\pi\left[ D_{\tau_n}^+ \right] \leq n + o(n).  
\end{equation}
Dividing by $n$, letting $n\ra\infty$ and then letting $k\ra\infty$ we obtain that 
\[
 \E_\pi[D_\infty^0] = \lim_{k\ra\infty} \E_\pi[ D_{\tau_k}^0 ] \leq 1.
\]
\end{proof}

\section{A zero-one law}\label{sec:01}

We showed under minimal assumptions in Theorem \ref{thm:ergrt}\ref{K01law} that $\P_\eta( \mathcal{R} \cup \mathcal{T}_+) = 1$. That is, the excited random walk is always either recurrent or transient to the right. 
In this section, we will show that with the addition of Assumptions \ref{asmiid} and \ref{asmell} that we can prove the following stronger 0-1 law for recurrence and transience. We will prove the 0-1 law under both the measures $\P_\eta$ and $\P_\pi$. 
\begin{prop}\label{01lemcor}
Let Assumptions \ref{asmMmu}--\ref{asmell} hold. Then, $\P_\eta( \mathcal{T}_+ ) = 1-\P_\eta(\mathcal{R}) \in \{0,1\}$ and $\P_\pi( \mathcal{T}_+ ) = 1-\P_\pi(\mathcal{R}) \in \{0,1\}$, where $\pi$ is the stationary measure for the cookie environment process from Corollary \ref{cor:invariant}.
\end{prop}

\begin{rem}
 It should be noted that the proof of the corresponding 0-1 law for nearest neighbor excited random walks is significantly easier. 
In fact, if we assume a strong form of ``uniform ellipticity'' then many of the technical difficulties in this section may be avoided. 
In particular, the proof of the 0-1 law for excited random walks in \cite[Lemma 9 and Proposition 10]{zERWZdstrip} can be used if we assume that there exist constants $c,C>0$ such that $c\mu(z) \leq \w_{x,j}(z) \leq C \mu(z)$ for all $x\in \Z$ and $j\geq 1$. 
\end{rem}


To prepare for the proof of Proposition \ref{01lemcor}
we begin by noting that Assumption \ref{asmmu} implies that there exists an integer $K_0 > B$ such that for any $K\geq K_0$ the simple random walk with jump distribution $\mu$ is irreducible restricted to $[0,K-1]$. That is, for $K\geq K_0$ and $x,y \in [0,K-1]$ there exists a finite path $\mathbf{z} = (z_0,z_1,\ldots,z_m)$ beginning at $z_0 = x$, ending at $z_m=y$, with $z_k \in [0,K-1]$ and $\mu(z_k-z_{k-1})>0$ for $k=1,2,\ldots,m$. 

Secondly, there exists an integer $K_1 \geq B$ such that for any $K\geq K_1$ we have
\[
 c_K^-:= E_\eta\left[ \prod_{j=1}^M \sum_{z=-K}^{0} \w_{0,j}(z) \right] > 0
\quad \text{and}\quad
 c_K^+:= E_\eta\left[ \prod_{j=1}^M \sum_{z=0}^K \w_{0,j}(z) \right] > 0. 
\]
Note that $c_K^-$ and $c_K^+$ are the probability that on the first $M$ visits to a fixed site all jumps are in $[-K,0]$  or $[0,K]$, respectively. 

\begin{lem}\label{lem:Kint}
Let Assumptions \ref{asmMmu}--\ref{asmell} hold. Then, for any $K \geq \max\{K_0,K_1\}$
\begin{equation}\label{Kintlb}
 \P_\eta\left( \tau_{2K} < \s_{-1}, \, X_{\tau_{2K}} = 2K, \, L_{\tau_{2K}}(x) \geq M, \, \forall x \in [K,2K-1] \right) > 0. 
\end{equation}
\end{lem}
\begin{proof}
Let $K\geq \max\{K_0,K_1\}$ be fixed, and for each $i=0,1,2,\ldots,K-1$ let $$\mathbf{x}^{(i)} = (x^{(i)}_0,x^{(i)}_1,\ldots,x^{(i)}_{m_i})$$ be a path of finite length such that 
\begin{itemize}
 \item $\mu(\mathbf{x}^{(i)}) := \prod_{k=1}^{m_i} \mu( x_k^{(i)} - x_{k-1}^{(i)} ) > 0$. That is, the path $\mathbf{x}^{(i)}$ has positive probability for the simple random walk with jump distribution $\mu$. 
 \item The path $\mathbf{x}^{(i)}$ begins at $x_0^{(i)} = i$ and ends at $x_{m_i}^{(i)} = K$. 
 \item $x_k^{(i)} \in [0,K-1]$ for all $k<m_i$, and $\sum_{k=0}^{m_i-1} \ind{x_k^{(i)} = x} \geq 1$ for all $x \in [0,K-1]$. That is, the path stays inside of $[0,K-1]$ up until the last step and the path visits every site in $[0,K-1]$ at least once. 
\end{itemize}
Now, we explain a strategy to force the event in the probability in \eqref{Kintlb} to occur. 
\begin{enumerate}
 \item When the random walk reaches a site $x \in [0,K-1]$
\begin{enumerate}
\item if there remains at least one cookie then the walk jumps to a site in $[x,x+K]$,
\item and if there are no cookies remaining then the walk jumps to a site in $[x+1,x+B]$. 
\end{enumerate}
 \item If the random walk is at a site $x=K+i \in [K,2K-1]$ with no cookies remaining, then the the walk begins following the path $K+\mathbf{x}^{(i)}$ until it either (a) reaches a site with at least one cookie remaining, or (b) reaches $2K$ (in which case there must have been no cookies remaining in $[K,2K-1]$ since the path visits every site in $[K,2K-1]$). 
 \item Whenever the random walk reaches a site $x \in [K,2K-1]$ with at least one cookie remaining, then the following step of the excited random walk is in $[x-K,x]$.
\end{enumerate}
Note that in this procedure, each time the random walk attempts to follow the path $K+\mathbf{x}^{(i)}$ as in the third step outlined above, the walk either succeeds in following the entire path or one cookie is removed from the interval $[K,2K-1]$. Therefore, at least one of the first $M K$ such attempts must succeed. This also limits the number of visits that the walk has at each site in $[0,K-1]$ by time $\tau_{2K}$. From the above outlined strategy we obtain that 
\begin{align*}
& \P_\eta\left( \tau_{2K} < \s_{-1}, \, X_{\tau_{2K}} = 2K, \, L_{\tau_{2K}}(x) \geq M, \, \forall x \in [K,2K-1] \right) \\
& \qquad \geq \left\{ \mu([1,B])^{(MK+1)} c_K^-c_K^+ \right\}^K \left( \prod_{i=0}^{K-1} \mu(\mathbf{x}^{(i)}) \right)^{MK} > 0.
\end{align*}
\end{proof}
As a corollary of the above lemma, we get that almost surely at some point the random walk will reach a new maximum and at that time have visited all $K$ of the sites immediately to the left at least $M$ times (so that no cookies remain in the $K$ sites immediately to the left and none of the sites to the right have been visited yet). 
\begin{cor}\label{lem:Lint}
 Let Assumptions \ref{asmMmu}--\ref{asmell} hold. For any $K \geq \max\{K_0,K_1\}$
\[
 \P_\eta\left( \exists n<\infty: X_{\tau_n} = n \text{ and }  L_{\tau_n}(x) \geq M, \, \forall x \in [n-K,n-1] \right) = 1.  
\]
Moreover, the same result holds for the measure $\P_\pi$. 	
\end{cor}
\begin{proof}
For $K \geq \max\{K_0,K_1\}$ fixed, define a sequence of stopping times $\{t_k\}_{k\geq 0}$ by
\[
 t_0 = 0 \quad\text{and}\quad
t_{k+1} = \inf\{ n > t_k: X_n \geq X_{t_k} + 2K \} = \tau_{X_{t_k} +2K}, 
\]
and let $\{E_k\}_{k\geq 0}$ be the sequence of events 
\[
 E_k = \left\{ X_{t_{k+1}} = X_{t_k} + 2K \text{ and } L_{t_{k+1}}(x) \geq M, \, \forall x \in [X_{t_k} + K, X_{t_k}+2K-1] \right\}. 
\]
If we let $\e_K>0$ be the probability in \eqref{Kintlb}, then it follows that 
\[
 \P_\eta\left( E_k \, | \s( X_i, \, i\leq t_k ) \right) \geq \e_K > 0, \quad \forall k\geq 1,
\]
since conditioned on $\s(X_i, \, i\leq t_k)$ the cookie environment to the right of $X_{t_k}$ is i.i.d.\ with the same distribution as the original distribution on cookie environments to the right of the origin. 
From this we can conclude that 
\[
 \P_\eta\left( \bigcup_{k=0}^{n-1} E_k \right) \geq 1-(1-\e_K)^n. 
\]
Taking $n\ra\infty$ we see that $\P_\eta\left( \bigcup_{k=0}^\infty E_k \right) = 1$.

The same proof works for the measure $\P_\pi$ since the cookie environment is also i.i.d.\ to the right of the origin under the stationary measure $\pi$. 
\end{proof}

For the following lemma, we introduce the notation $\hat\w(I) = \a(\w, M \mathbf{1}_I(\cdot))$ to denote the cookie environment $\w$ with all cookies removed from the set $I \subset \Z$. 
Also, let 
\[
 B_0 = \inf \left\{ K\geq 0: \eta\left( \sum_{z=K+1}^\infty \w_{x,j}(z) = 0 \, \forall x \in \Z, j\geq 1 \right) = 1 \right\}, 
\]
be the uniform upper bound on the maximum jump size of the excited random walk. Note that it may be the case that $B_0 = \infty$. However, in the case of uniformly bounded above jumps ($B_0 < \infty$) we can prove the following. 

\begin{lem}\label{lem:LL}
If $B_0 < \infty$ and $L > K \geq \max\{ B_0, K_0 \}$ then 
then
\begin{equation}\label{trnbL}
 P_{\hat{\w}([-K,-1])}( \mathcal{T}_+ \cap \{ \s_{-L} = \infty \} ) > 0 \quad\Longrightarrow \quad
P_{\hat{\w}([-L,-1])}( \mathcal{T}_+ \cap \{ \s_{-L} = \infty \} ) > 0.
\end{equation}
\end{lem}
\begin{proof}
 If the probability on the left side of \eqref{trnbL} is positive, then there exists a finite path $\mathbf{x} = (x_0,x_1,\ldots,x_m)$ starting at $x_0 = 0$ which stays in $(-L,\infty)$ and ends at $x_m \geq 0$ such that 
\begin{equation}\label{initialpathL}
 P_{\hat{\w}([-K,-1])} \left( X_{[0,m]} = \mathbf{x}, \, \inf_{k> m} X_k \geq 0 \right) > 0. 
\end{equation}
Now the path $\mathbf{x}$ may no longer have positive probability in the cookie environment $\hat\w([-L,-1])$. 
However, since $K \geq B_0$ any excursion to the left of the origin must exit the negative half-line from a site in $[-B_0,-1] \subset [-K,-1]$, and since $L>K\geq K_0$ then every excursion in the path $\mathbf{x}$ to the left of the origin can be replaced by an excursion that stays in $(-L,-1]$, only uses jumps supported by the distribution $\mu$ and 
exits the negative half-line in the same manner as the original excursion.
Therefore, there exists a path 
$\mathbf{y} = (y_0,y_1,\ldots,y_\ell)$ that has positive probability in the cookie environment $\hat\w([-L,-1])$, ends at $y_\ell = x_m$ and which also has the same number of visits to all sites right of the origin. 
From this we see that \eqref{initialpathL} implies that 
\[
  P_{\hat{\w}([-L,-1])} \left( X_{[0,\ell]} = \mathbf{y} , \, \inf_{k> \ell} X_k \geq 0 \right) > 0. 
\]
Clearly, this implies that the probability on the right side of \eqref{trnbL} is also positive. 
\end{proof}

\begin{lem}\label{01lem}
Let $\mathcal{P} = \{ X_n \geq X_0, \, \forall n\geq 0\}$ be the event that the excited random walk never goes to the left of its initial location. 
If Assumptions \ref{asmMmu}--\ref{asmell} hold and 
$\P_{\eta}( \mathcal{T}_+ ) > 0$, then $\P_{\eta}( \mathcal{T}_+ \cap \mathcal{P} ) > 0$. 
\end{lem}

\begin{rem}
 As noted in the proof of Proposition \ref{limsupprop}, due to Assumption \ref{asmMmu} if the excited random walk stays non-negative then it necessarily is transient to the right. That is $\P_\eta( \mathcal{T}_+ \cap \mathcal{P} ) = \P_\eta(\mathcal{P})$. 
\end{rem}

\begin{proof}
We divide the proof into two cases, depending on whether the jumps are uniformly bounded above or not. 

\noindent\textbf{Case I: Unbounded jumps ($B_0 = \infty$).}
If the event $\mathcal{T}_+$ occurs, then at some point the walk is always to the right of the origin. Therefore, there exists a finite path $\mathbf{x} = (x_0,x_1,\ldots,x_m)$ with $x_0=0$ and $x_m \geq 0$ such that 
\begin{equation}\label{initialpath}
 \P_\eta\left( X_k = x_k,\, k\leq m, \text{ and } \inf_{k>m} X_k \geq 0 \right) > 0. 
\end{equation}
Let $L<\infty$ be such that the finite path $\mathbf{x}$ is contained in $(-L,\infty)$. 
Since the jumps are unbounded, there exists a $j\leq M$ such that $\eta( \sum_{z \geq L} \w_{0,j}(z) > 0 ) > 0$. 
It follows from Lemma \ref{lem:Kint} above that for $K \geq \max{K_0,K_1}$ there exists a path 
$\mathbf{y} = (y_0,y_1,\ldots,y_\ell)$ starting at $y_0=0$ with $y_k \in [0,2K]$ for $k<\ell$ and ending at $y_\ell \geq 2K+L$ with $ \P_\eta( X_k = y_k, \, \forall k \leq \ell ) > 0$. 
Indeed, after forcing the event in Lemma \ref{lem:Kint} to occur, with positive probability the random walk then performs $j-1$ loops from $2K$ back to $2K$ and contained in $[K,2K]$ and then on the $j$-th visit to $2K$ jumps at least $L$ to the right. 
Then, by appending the path $\mathbf{y}$ to the beginning of the path $\mathbf{x}$ and using 
the assumption that the cookie environment is i.i.d.\ (Assumption \ref{asmiid})
we get that 
\begin{align*}
 \P_\eta(\mathcal{P} \cap \mathcal{T}_+) 
&\geq \P_\eta\left( X_{k} = \begin{cases} y_k & k\leq \ell \\ y_\ell + x_{k-\ell} & \ell < k \leq \ell + m, \end{cases} \text{ and }  \inf_{k > \ell + m } X_k \geq y_\ell \right) \\
&=  \P_\eta( X_k = y_k, \, \forall k \leq \ell ) \P_\eta\left( X_k = x_k,\, k\leq m, \text{ and } \inf_{k>m} X_k \geq 0 \right) > 0. 
\end{align*}

\noindent\textbf{Case II: Uniformly bounded jumps ($B_0 < \infty$).}
For $K<\infty$ let  $\rho_K$ be a stopping time defined by 
\[
 \rho_K = \inf\left\{ \tau_n: X_{\tau_n} = n \text{ and }  L_{\tau_n}(x) \geq M, \, \forall x \in [n-K,n-1] \right\}.
\]
Note that Lemma \ref{lem:Lint} implies that $\P_\eta( \rho_K < \infty) = 1$ for any $K\geq \max\{K_0,K_1\}$. 
Therefore, if $\P_\eta(\mathcal{T}_+) > 0$ then there exists an $L>K$ such that 
\[
 \P_\eta\left( \mathcal{T}_+ \cap \left\{ \inf_{k > \rho_K} X_k >  X_{\rho_K} - L \right\} \right) > 0,
\]
and then it follows from Lemma \ref{lem:LL} that 
\begin{equation}\label{rhoLnbL}
 \P_\eta\left( \mathcal{T}_+ \cap \left\{ \inf_{k > \rho_{L}} X_k >  X_{\rho_{L}} - L \right\} \right) > 0.
\end{equation}
Thus, the excited random walk with all cookies removed from $[-L,-1]$ never reaches $-L$ with positive probability. 
Since Lemma \ref{lem:Kint} implies that with positive probability there is a non-negative path that ends at $X_{\tau_{2L}} = 2L$ with all cookies used in $[L,2L-1]$, we can thus conclude that after this happens there is then a positive probability of never going back below $L$ and thus $\P_\eta( \mathcal{P} \cap \mathcal{T}_+ ) > 0$. 
\end{proof}

A corollary of the above lemma is that the same result holds for the measure $\P_\pi$. 
\begin{cor}\label{01lempi}
Let Assumptions \ref{asmMmu}--\ref{asmell} hold. 
If $\pi$ is the stationary measure for the cookie environment process from Corollary \ref{cor:invariant} and $\P_{\pi}( \mathcal{T}_+ ) > 0$, then $\P_{\pi}( \mathcal{T}_+ \cap \mathcal{P} ) > 0$. 
\end{cor}

\begin{proof}
Since $\pi$ and $\eta$ have the same distribution on $\{\w_x \}_{x\geq 0}$, then $\P_\eta(\mathcal{T}_+ \cap \mathcal{P}) = \P_\pi(\mathcal{T}_+ \cap \mathcal{P})$. 
Thus, by Lemma \ref{01lem} it is enough to show that $\P_\pi(\mathcal{T}_+) \leq \P_\eta(\mathcal{T}_+)$. 

For two cookie environments let $(\w,\w')$ be the cookie environment which agrees with $\w$ to the left of the origin and $\w'$ to the right of the origin. 
That is, $(\w,\w')_{x,j} = \w_{x,j} \ind{x<0} + \w'_{x,j} \ind{x\geq 0}$. 
For $(\w,x) \in \Psi$ define
\[
 R(\w,x) = \int P_{(\w,\w')}^x( \mathcal{T}_+ ) \, \eta(d\w'). 
\]
Thus $R(\w,x)$ is the probability the walk started at $x$ is transient to the right conditioned on the realization of the cookie environment to the left of the origin. 
We claim that the mapping $(\w,x) \ra R(\w,x)$ is lower semi-continuous. 
To see this, note that if the walk is transient to the right then there is a last visit to the negative integers by the walk. Therefore, if $\Pi_x$ denotes the collection of paths $\mathbf{x} = (x_0,x_1,\ldots,x_m)$ of finite length with $x_0 = x$, $x_{m-1}<0$ and $x_m \geq 0$, then 
\begin{equation}\label{Rsum}
 R(\w,x) = 
\P_\eta^x( X_n \geq 0, \, \forall n\geq 0) + 
\sum_{ \mathbf{x} \in \Pi_x } \int P_{(\w,\w')}^x\left( X_k = x_k, \, k\leq m, \, \inf_{k>m} X_k \geq 0 \right)\, \eta(d\w'). 
\end{equation}
For each fixed path $\mathbf{x} \in \Pi_x$, the probabilities inside the integral on the right only depend on finitely many of the coordinates $\w_{y,j}$ from the cookie environment $\w$ and thus it is easy to see that the mapping $\w \mapsto \int P_{(\w,\w')}^x\left( X_k = x_k, \, k\leq m, \, \inf_{k>m} X_k \geq 1 \right)\, \eta(d\w')$ is continuous for any fixed $\mathbf{x} \in \Pi_x$. Then it follows from \eqref{Rsum} that $R(\w,x)$ is lower semi-continuous, and since the measure $\pi$ was constructed as the weak limit of the measures $\mu_{n_m}$ we can thus conclude that 
\begin{equation}\label{PpiR}
\P_\pi(\mathcal{T}_+) = E_\pi\left[ R(\w,X_0) \right]
\leq \liminf_{m\ra\infty} E_{\mu_{n_m}} \left[ R(\w,X_0) \right].
\end{equation}
However, for any $n\geq 1$
\begin{align}
 E_{\mu_n}\left[  R(\w,X_0) \right]
&= \frac{1}{n} \sum_{k=1}^n \E_\eta\left[ R(\th^k \bar{\w}(\tau_k), X_{\tau_k} -k) \right] \nonumber \\
&= \frac{1}{n} \sum_{k=1}^n \E_\eta\left[ \P_\eta\left( \mathcal{T}_+ \, | \, \s(\w_x, \, X_i: \, x < k, \, i\leq \tau_k \right) \right] \nonumber \\
&= \P_\eta( \mathcal{T}_+ ). \label{EmuR}
\end{align}
Therefore, it follows from \eqref{PpiR} and \eqref{EmuR} that $\P_\pi(\mathcal{T}_+) \leq \P_\eta(\mathcal{T}_+)$.
\end{proof}

We are now ready to give the proof of the main result of this section. 
\begin{proof}[Proof of Proposition \ref{01lemcor}]
Let $\b = \P_\eta( \mathcal{P} ) = \P_\eta( \mathcal{P} \cap \mathcal{T}_+ )$. 
If $\P_\eta( \mathcal{T}_+ ) > 0$, then Lemma \ref{01lem} implies that $\b > 0$. 
The proof of the proposition will be finished by showing that $\b>0$ in turn implies that $\P_\eta( \mathcal{T}_+ ) = 1$. 
To this end, let $0 = N_0 < B_1 \leq N_1 \leq B_2 \leq N_2 \leq B_3 \leq \ldots$ be a sequence of stopping times for the random walk defined as follows. 
\begin{equation}\label{BjNjdef}
 B_j = \inf \{ n > N_{j-1}: X_n < X_{N_{j-1}} \}
\quad \text{and}\quad
 N_j = \inf \left\{ n > B_j: X_n > \max_{k\leq B_j} X_k \right\}, \quad j \geq 1.  
\end{equation}
(We will use the convention that $B_j = \infty$ implies that $N_k = B_k = \infty$ for all $k> j$.)
Since $\P_\eta( \limsup_{n\ra\infty} X_n = +\infty) = 1$, we can conclude that 
\begin{equation}\label{PNB}
 \P_\eta( N_j < \infty | \, B_j < \infty) = 1, \qquad j\geq 1. 
\end{equation}
Also, at a stopping time $N_{j-1}$ the random walk is at a new maximum and so the environment to the right of the current location again has the same distribution as the environment to the right of the origin under the measure $\eta$. Thus, 
\begin{equation}\label{PBN}
 \P_\eta( B_j = \infty | \, N_{j-1} < \infty) = \b, \qquad j\geq 1. 
\end{equation}
From these two facts we can conclude that $\P_\eta(B_j < \infty) = (1-\b)^j$ for all $j\geq 1$. Therefore, if $\b > 0$ it follows that $B_j = \infty$ for some $j\geq 1$ almost surely, 
and thus that $\liminf_{n\ra\infty} X_n \geq 0$ with probability 1. From Theorem \ref{thm:ergrt}\ref{K01law} we can conclude
that $\b>0$ implies that $\P_\eta( \mathcal{T}_+) = 1$.

To prove the conclusion of the Proposition for $\P_\pi$, we first need to show that 
\begin{equation}\label{PpiTR}
 \P_\pi(\mathcal{T}_+) = 1-\P_\pi(\mathcal{R})
\end{equation}
To this end, first note that since $\P_\pi(\limsup_{n\ra\infty} X_n = \infty) = 1$ it follows that 
\[
 \P_\pi\left( \liminf_{n\ra\infty} X_n < \infty \right) = 1-\P_\pi(\mathcal{T}_+). 
\]
Since the cookie environment to the right of the origin has the same distribution under $\pi$ and $\eta$, the same argument in the proof of Lemma \ref{infsupR} shows that
\[
 \P_\pi\left( \liminf_{n\ra\infty} X_n < \infty \right) = \P_\pi\left( \sum_{n\geq 0} \ind{X_n = x} = \infty, \, \forall x\geq 0 \right).
\]
That is, if the walk is not transient to the right under the measure $\P_\pi$ then the walk visits every site to the right of the origin infinitely often. 
Finally, note that since $\Omega_{M,\mu}:= \left( M_1(\Z)^M \times \{\mu\}^{\N} \right)^{\Z}$ is a closed subset of $\Omega$,\footnote{Note that $\Omega_{M,\mu}$ differs from $\Omega_{M,\mu}^+$ in that we no longer require the first $M$ cookies at each site to have non-negative drift.} it follows from Assumption \ref{asmMmu} that the stationary distribution $\pi$ is concentrated on $\Omega_{M,\mu} \times \Z_+$. For all cookie environments in $\Omega_{M,\mu}$ it is easy to see that if any site is visited infinitely often then all sites are visited infinitely often. 
This completes the proof of \eqref{PpiTR}. 	

Having proved \eqref{PpiTR}, it follows from Corollary \ref{01lempi} we need only to show that $\b > 0$ implies $\P_\pi(\mathcal{T}_+) = 1$. 
Since \eqref{PNB} and \eqref{PBN} also hold for the measure $\P_\pi$, the same argument as above shows that $\b>0$ implies $\P_\pi(\liminf_{n\ra\infty} X_n \geq 0) = 1$ which by \eqref{PpiTR} implies 
that $\P_\pi(\mathcal{T}_+) = 1$. 
\end{proof}

\section{A sharp criterion for recurrence/transience}\label{sec:sharp}

In this section we give the proof of the sharp criterion for recurrence/transience from Theorem \ref{thm:rt}. 
As a first step we prove the following lemma. 

\begin{lem}\label{EpiDinf1}
 If $\P_\pi(\mathcal{T}_+) = 1$, then $\E_\pi[ D_\infty^0 ] =1 $. 
\end{lem}
\begin{proof}
Since Lemma \ref{EpiDinf} implies that $\E_\pi[D_\infty^0] \leq 1$, we only need to show that $\E_\pi[D_\infty^0] \geq 1$. 
 First of all, note that for any $n,m\geq 1$ and $-m<x<n$ that 
 \begin{align*}
  x &= E_\w^x[ X_{\tau_n \wedge \s_{-m}} ] - E_\w^x[ D_{\tau_n \wedge \s_{-m}} ] \\
  &\geq n P_\w^x( \tau_n < \s_{-m} ) - \sum_{y=-m+1}^{-1} \bar{d}(\w_y) - \sum_{y=0}^{n-1} E_\w^x[ D_{\infty}^y ]. 
 \end{align*}
Taking expectations of $(\w,x)$ with respect to the stationary measure $\pi$ we obtain that 
\begin{align*}
 n \P_\pi( \tau_n < \s_{-m}) \leq \E_\pi[X_0 \ind{X_0 < n}] + (m-1)\d + n \E_\pi[ D_\infty^0].
\end{align*}
Therefore, dividing by $n$ and letting $n\ra\infty$ with $m$ fixed we can conclude that 
\[
 \P_\pi( \s_{-m} = \infty) = \lim_{n\ra\infty} \P_\pi( \tau_n < \s_{-m} ) \leq \E_\pi[D_\infty^0]. 
\]
Since we are assuming that $\P_\pi(\mathcal{T}_+) = 1$, then $\lim_{m\ra\infty} \P_\pi(\s_{-m} = \infty) = 1$. 
\end{proof}

Next, we prove that the criterion for recurrence/transience holds for the stationary measure $\P_\pi$ in place of $\P_\eta$. 
\begin{prop}\label{Ppirt}
 Let Assumptions \ref{asmMmu}--\ref{asmell} hold, and let $\pi$ be the stationary measure for the cookie environment process from Corollary \ref{cor:invariant}. 
\begin{enumerate}
 \item If $\d>1$ then $\P_\pi(\mathcal{T}_+) = 1$. 
 \item If $\d\leq 1$ then $\P_\pi(\mathcal{R}) = 1$. 
\end{enumerate}
\end{prop}

\begin{proof}
It follows from Proposition \ref{01lemcor} that either $\P_\pi(\mathcal{R}) = 1$ or $\P_\pi(\mathcal{T}_+) = 1$. 
First, assume that $\P_\pi(\mathcal{R}) = 1$. Since the random walk then visits the origin infinitely often, we can conclude that $\E_\pi[ D_\infty^0] = \E_\pi[ \bar{d}(\w_0) ] = \E_\eta[ \bar{d}(\w_0) ] = \d$. Therefore, it follows from Lemma \ref{EpiDinf} that $\P_\pi(\mathcal{R}) = 1$ implies $\d\leq 1$.

Conversely, assume that $\P_\pi(\mathcal{T}_+) = 1$. 
Then, Lemma \ref{EpiDinf1} implies that $\d \geq \E_\pi[D_\infty^0] = 1$. 
To improve this weak inequality to a strict inequality we need to consider two cases. 

\noindent\textbf{Case I:} $\P_\pi(X_0 > 0) > 0$. 
 In this case, 
\[
 \P_\pi( X_n > 0, \forall n\geq 0) \geq E_\pi\left[ P_\w^{X_0}( \mathcal{P} \cap \mathcal{T}_+ ) \ind{X_0 > 0} \right] 
= \pi( X_0 > 0 ) \P_\eta( \mathcal{P} \cap \mathcal{T}_+ ) > 0, 
\]
from which it follows that $\d = \E_\pi[\bar{d}(\w_0)] > \E_\pi[D_\infty^0] = 1$.

\noindent\textbf{Case II:} $\P_\pi(X_0 = 0) = 1$. 
In this case it must be that the first cookie cannot induce a jump to the right larger than 1, and since Assumption \ref{asmell} implies that there is always a positive probability of a nonpositive jump on the first visit to a site we can conclude that $E_\eta[ d(\w_{0,1}) ] < 1 = \E_\pi[D_\infty^0] \leq \E_\eta[\bar{d}(\w_0)]$. We may thus conclude that $\eta( d(\w_{0,1})  < \bar{d}(\w_0) ) > 0$, and thus
\begin{align*}
 \P_\pi( D_\infty^0 < \bar{d}(\w_0) )
&\geq \P_\pi\left( d(\w_{0,1}) < \bar{d}(\w_0), \, X_0 = 0, \, \inf_{n\geq 1} X_n \geq 1 \right) \\
&= E_\eta\left[ \ind{ d(\w_{0,1}) < \bar{d}(\w_0) } \w_{0,1}(1) \right] \P_\eta( \mathcal{P} \cap \mathcal{T}_+ ) \\
&> 0,  
\end{align*}
where the last equality follows from Lemma \ref{01lem} and the fact that $\w_{0,1}(1) > 0$ due to Assumption \ref{asmell}. 
It then follows that  $1 =\E_\pi[ D_\infty^0 ] < \E_\pi[ \bar{d}(\w_0) ] = \d$.
\end{proof}

We conclude by giving the proof of Theorem \ref{thm:rt}, which follows easily from Proposition \ref{Ppirt} and the results from Section \ref{sec:01}. 

\begin{proof}
Due to Theorem \ref{thm:ergrt}\ref{K01law} and Proposition \ref{01lemcor}, we need only to show that $\P_\eta(\mathcal{T}_+) \iff \d>1$. 

\begin{align*}
\d > 1 &\iff \P_\pi(\mathcal{T}_+) = 1 && \text{(Proposition \ref{Ppirt})} \\
&\iff \P_\pi(\mathcal{P} \cap \mathcal{T}_+) > 0 && \text{(Corollary \ref{01lempi} and Proposition \ref{01lemcor})} \\
&\iff \P_\eta(\mathcal{P} \cap \mathcal{T}_+) > 0 && \text{(Corollary \ref{cor:invariant}\ref{piprop1})}\\
&\iff \P_\eta( \mathcal{T}_+ ) = 1. && \text{(Lemma \ref{01lem} and Proposition \ref{01lemcor})}
\end{align*}
\end{proof}

\bibliographystyle{alpha}
\bibliography{CookieRW}

\end{document}